\documentclass[10pt,psamsfonts]{amsart}
\usepackage[margin=3.6cm]{geometry}
\usepackage[english]{babel}
\usepackage{amsmath,amssymb,graphicx,enumerate,bbm,caption}
\usepackage{float}
\usepackage{array}
\usepackage{dsfont}
\usepackage{xcolor}
\usepackage{booktabs}
\usepackage{multirow}

\usepackage{amsmath}
\usepackage{amsthm}
\usepackage{amssymb}
\usepackage{amscd}
\usepackage{amsfonts}
\usepackage{amsbsy}
\usepackage{epsfig,afterpage}
\usepackage{psfrag}
\usepackage{overpic}
\usepackage{latexsym}
\usepackage{pstcol}
\usepackage{graphicx}
\usepackage{mathtools}
\usepackage[all]{xy}
\usepackage{float}

\renewcommand{\arraystretch}{1.5}
\newcommand{\R}{\ensuremath{\mathbb{R}}}
\newcommand{\C}{\ensuremath{\mathbb{C}}}

\newcommand{\e}{\varepsilon}

\newtheorem {theorem} {Theorem} [section]
\newtheorem {proposition}  {Proposition}[section]
\newtheorem {corollary}  {Corollary}[section]

\newtheorem {definition}  {Definition}[section]
\newtheorem {remark} {Remark}

\newtheorem {example} {Example}

\begin{document}

\title[complex potentials]
{complex potentials and holomorphic differential equations}

\author[G.Rondon and P. R. Silva]
{Gabriel Rondon$^1$ and Paulo Ricardo da Silva$^2$}

\address{$^{1}$ Departamento de Matemática, Instituto de Ciências Exatas (ICEx), Universidade Federal de Minas Gerais (UFMG), Av. Pres. Antônio Carlos, 6627, Cidade Universitária - Pampulha, 31270-901, Belo Horizonte, MG, Brazil}
 \address{$^2$ Departamento de Matem\'{a}tica, Instituto de Bioci\^{e}ncias Letras
	e Ci\^{e}ncias Exatas, UNESP, Univ Estadual Paulista, Rua C. Colombo, 2265,
	CEP 15054--000 S\~{a}o Jos\'{e} do Rio Preto, S\~{a}o Paulo, Brazil}
\email{grondon@ufmg.br}	
\email{ paulo.r.silva@unesp.br}


\date{}
\dedicatory{}
\maketitle

\begin{abstract}
A complex potential is a holomorphic function $\Omega:\mathbb{C} \to \mathbb{C}$ whose real and 
imaginary parts generate a pair of orthogonal foliations, representing the equipotential lines and the 
streamlines of  $\dot{z} = \overline{\Omega'(z)}$. 
In this work, we generalize the concept of potential to the broader class of  dynamical systems of the 
form $\dot{z} = f(z)$, with $f:\mathbb{C} \to \mathbb{C}$ holomorphic. The resulting potential induces 
a rectification mapping providing a natural framework for the topological classification of phase portraits 
of planar polynomial vector fields. The existence of complex potentials serves as a powerful tool in addressing fundamental problems, such as  the establishment of bounds for the number of limit cycles in piecewise-smooth systems, and the local configuration of curvature lines around umbilic points, among others.
\end{abstract}

\section{Introduction}

The interplay between complex analysis and the qualitative theory of planar differential equations has proven to be a rich source of insights, linking analytic properties of functions with the geometric structure of phase portraits. 

We recall that a \textit{holomorphic function} $f$ is a complex-valued function defined in a domain $V \subseteq \mathbb{C}$ where $u=\Re(f)$ and $v=\Im(f)$ are continuous, the first order partial derivatives exist in $V$ and satisfy the \textit{Cauchy--Riemann equations}
    \[u_x = v_y, \qquad u_y = -v_x, \qquad \forall z=x+iy \in V.\]
The Looman--Menchoff Theorem states that these conditions are sufficient to guarantee the analyticity of $f$. If $f$ is holomorphic in a punctured disc $D(z_0,R)\setminus \{z_0\}$ but not differentiable at $z_0$, we say that $z_0$ is a \textit{singularity} of $f$; in this case $f$ admits a Laurent expansion
\begin{equation}\label{eq:Laurent}
    f(z) = \sum_{k=1}^{\infty}\frac{B_k}{(z-z_0)^k} + \sum_{k=0}^{\infty} A_k (z-z_0)^k.
\end{equation}

When a planar vector field is identified with a holomorphic function $f : \mathbb{C} \to \mathbb{C}$, several remarkable features emerge: the absence of limit cycles, the reversibility of polynomial systems with real coefficients, and the non‑existence of the center–focus problem, among others. These facts have been thoroughly discussed in previous works \cite{BT,GGX,GSR1}. 

Throughout the paper we study \textit{planar vector fields} $f=(u,v)$ seen as functions of a single complex variable $f = u+ i v$ in a simply connected domain $D \subseteq \mathbb{C}$. A \textit{trajectory} of $ \dot z = f(z)$ is a curve $z(t)=x(t)+iy(t)$, $t\in\mathbb{R}$, which satisfies the differential system
\begin{equation}\label{eq:system}
        \dot x = u(x,y), \quad \dot y = v(x,y).
\end{equation}
We will assume that either $f$ is a holomorphic function or that $f$ is an \textit{anti-holomorphic function}, that is, $\overline{f}$ is holomorphic.

When a vector field is regarded as a function of a complex variable, it is often possible to integrate it explicitly and, consequently, to determine its phase portrait. This is precisely the case that occurs in \textit{fluid dynamics}, where the trajectories are referred to as \textit{streamlines}. As an illustration, consider the anti-holomorphic system $\dot{z}=\overline{z}^2$. The corresponding real system is
\[x' = x^2-y^2, \qquad y' = -2xy,\]
which is nonlinear and not immediately solvable. However, the trajectories satisfy the exact equation $2xy\,dx+(x^2-y^2)\,dy=0$, which is the differential of $\psi(x,y)=x^2y-\frac{y^3}{3}$. Hence, the streamlines are precisely the level curves of $\psi$, as shown in Figure~\ref{figcv-d}. This example already hints at the existence of a first integral, a feature that will be explained systematically through the concept of a complex potential.

\begin{figure}[H]
	\centering 
	\includegraphics[width=8cm, height=4cm]{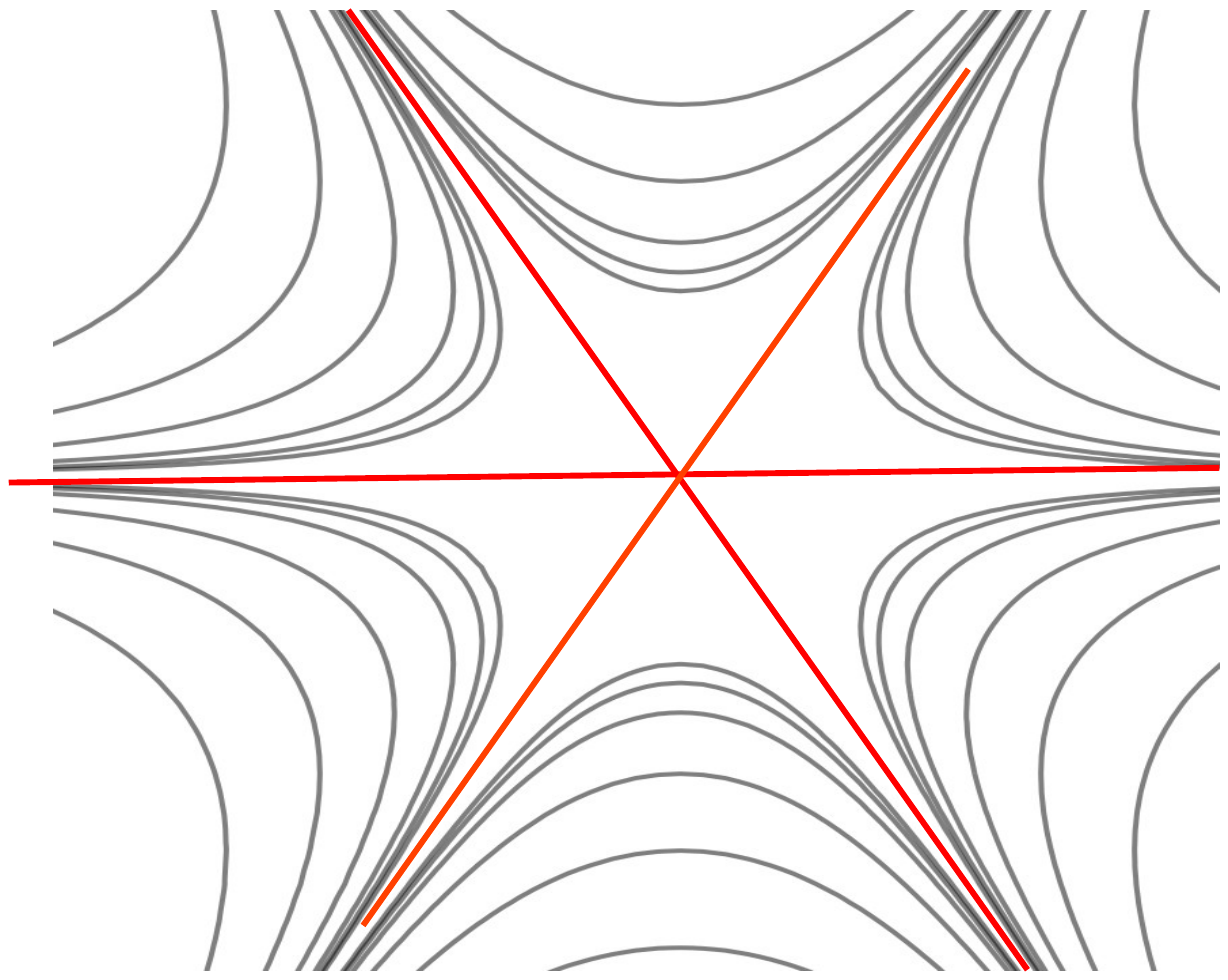} 
	\caption{Streamlines of $f(z)=\overline{z}^2$.}\label{figcv-d}
\end{figure}

Given a holomorphic system $\dot{z}= f(z)$ defined on some open set, its local phase portrait near a singularity is completely classified up to conformal conjugation. As shown in \cite{BT,GGX}, near a regular point the system is conjugate to $\dot{z}=1$, near a simple zero to $\dot{z}=z$, near a multiple zero of order $n$ to $\dot{z}=z^n/(1+\lambda z^{n-1})$ (where $\lambda=\operatorname{Res}(1/f,z_0)$), and near a pole of order $n$ to $\dot{z}=1/z^n$. This classification underscores the rigid normal form structure imposed by holomorphicity.

Both holomorphic systems $\dot{z}=f(z)$ and anti-holomorphic systems $\dot{z}=\overline{f(z)}$ (with $f$ holomorphic) exhibit strong integrability properties. For holomorphic systems, the trajectories are given implicitly by the level curves $\Im\int\frac{1}{f(z)}dz = \text{constant}$, which provides an explicit description of orbits though not necessarily a real-valued first integral. For anti-holomorphic systems, Theorem~\ref{teopsi} shows they are simultaneously gradient and Hamiltonian: if $\Omega(z) = \int f(z)\,dz$, then $\psi(x,y) = \Im\,\Omega(x+iy)$ is a genuine first integral. The example $\dot{z}=\overline{z}^2$ above corresponds to $\Omega(z)=z^3/3$, whose imaginary part is exactly $\psi(x,y)=x^2y-y^3/3$.

These integrability properties are intimately connected with the classical notion of a \emph{complex potential} from fluid dynamics. Given a holomorphic function $\Omega = \phi + i\psi$, the velocity field $\dot z = \overline{\Omega'(z)}$ describes an ideal incompressible irrotational flow, with $\psi$ serving as a stream function. The real and imaginary parts of $\Omega$ define orthogonal families, equipotential lines and streamlines, that coincide with the integral curves of $\nabla\phi$ and of the Hamiltonian field with first integral $\psi$. 

In this work we systematically exploit this connection. For a general holomorphic system $\dot z = f(z)$, we consider the generalized potential $\Omega(z) = \int \frac{1}{f(z)}\,dz$ (defined on a suitably cut domain), whose level curves $\Im\Omega=\text{constant}$ capture the phase portrait. For anti-holomorphic systems, the potential $\Omega(z)=\int f(z)dz$ directly provides a first integral. This unified framework explains not only the integrability of both classes but also practical consequences such as the algebraic nature of trajectories in polynomial anti-holomorphic systems (Theorem~\ref{teoalg}) and the existence of rectifying coordinates that trivialize the dynamics.

The piecewise-smooth setting breaks this rigidity in a controlled manner. By taking different holomorphic or anti-holomorphic systems on either side of a switching line, one obtains a discontinuous vector field that inherits explicit orbital descriptions or first integrals on each smooth piece. The mismatch between these structures can create isolated periodic solutions. This phenomenon has been previously explored for piecewise holomorphic systems in \cite{GSR2,GSR4,GSR3}, where bounds on the number of limit cycles were established. In the present work we extend this analysis to more general piecewise systems, including mixtures of holomorphic and anti-holomorphic parts, showing how complex potentials provide precise analytical control over the crossing conditions that determine the existence and number of limit cycles.

Our attention here is devoted to questions related to the integrability and global classification of holomorphic systems, and to applications in piecewise-smooth dynamics and differential geometry. The main contributions of this work are:
\begin{enumerate}
    \item \textbf{Generalization of the complex potential concept.} We introduce the concept of a complex potential for general holomorphic systems $\dot{z}=f(z)$ (Section \ref{sec:potencial}). We prove that anti-holomorphic systems are simultaneously \textit{gradient} and \textit{hamiltonian} (Theorem~\ref{teopsi}), and that polynomial potentials provide a class of systems whose trajectories are all algebraic (Theorem~\ref{teoalg}). We also present explicit potentials for the normal forms of holomorphic vector fields (Theorem~\ref{teofn}).

    \item \textbf{Circulation, net flow and complex time.} We show how the potential can be used to analyze circulation and net flow (Subsection \ref{sec:net1}), and to study the flow of systems with \textit{complex time}, i.e., $\frac{dz}{dT}=f(z)$, $T\in\mathbb{C}$ (Subsection \ref{sec:net2}). In this context, the real and imaginary parts of the potential act as first integrals for the real and imaginary time flows, respectively.

    \item \textbf{Global classification of cubic polynomial systems.} We classify the global phase portraits of cubic monic centered polynomial systems on the Poincaré disk (Theorem~\ref{teosepala}). The potential is used to define a rectifying coordinate system that trivializes the dynamics, revealing the decomposition into canonical regions (center-type, sepal-type, and $\alpha$--$\omega$ regions).

    \item \textbf{Bounds for limit cycles in a special class of piecewise smooth systems.} The complex potential framework yields sharp bounds for limit cycles in various piecewise-smooth configurations. For systems mixing anti-holomorphic and holomorphic parts across a switching line, we prove: piecewise linear systems have at most one limit cycle when the equilibrium of the holomorphic part lies on the switching line (Theorem~\ref{teo_a}), and at most three limit cycles when this equilibrium is arbitrary (Theorem~\ref{teo_b}). 

    For piecewise systems where both sides are anti-holomorphic, we obtain a complete hierarchy: piecewise linear systems admit no limit cycles, piecewise quadratic systems at most one, and piecewise cubic systems at most three (Theorem~\ref{thm:limit_cycles}). These bounds are sharp and demonstrate how the degree of the polynomials directly controls the maximum number of limit cycles. The proofs rely crucially on the first integrals provided by complex potentials on each smooth piece.

    \item \textbf{Connection with differential geometry.} We show how the local behavior of principal curvature lines around an isolated umbilic of a constant mean curvature (CMC) immersion can be determined by a monomial holomorphic function $Q(z)=z^n$ (Section~\ref{sec:curvature}). The phase portrait of $\dot{z}=\overline{z^n}$ precisely describes the configuration of the principal foliations, linking the index of the umbilic to the exponent $n$.
\end{enumerate}
The paper is organized according to the contributions listed above. Sections \ref{sec:potencial}–\ref{sec:net} develop the theory of complex potentials and their basic applications. Section \ref{sec:cubic} applies this framework to classify cubic polynomial systems. Section \ref{sec:cycles} derives bounds for limit cycles in piecewise-smooth systems using the potential approach. Finally, Section \ref{sec:curvature} presents the geometric application to principal curvature lines on CMC surfaces.

\section{Complex Potential}\label{sec:potencial}

A \textit{complex potential} $\Omega(z)=\phi(x,y)+i\psi(x,y)$ is a holomorphic function defined in a simply connected domain $D\subseteq\C$.
The \textit{complex velocity field} associated with $\Omega$ is the vector field $f(z)=\overline{\Omega'(z)}.$ 

A direct consequence of the Cauchy–Riemann equations is the following theorem.

\begin{theorem} \label{teopsi}
The complex velocity field associated with the  potential $\Omega=\phi+i\psi$ is simultaneously a gradient and a Hamiltonian vector field.
More specifically $\overline{\Omega'}=\nabla \phi$ and $ \psi$ is a first integral.
 \end{theorem}
\begin{proof} Consider the complex potential
$\Omega$  and  the complex velocity field associated $f$. Thus
$\Omega'=\overline{f}=\tilde{u}+i\tilde{v}$ is a holomorphic function
and
$\overline{\Omega'}=f=\tilde{u}+ i(-\tilde{v}).$
Using the Cauchy-Riemann equation for $\overline{f}$ we have
$\tilde{u}_y=(-\tilde{v})_x$
and thus we have that $f=\tilde{u}-i\tilde{v}$ is a gradient vector field. Besides
\[\tilde{u}+i\tilde{v}=\overline{f}=\Omega'=\phi_x+i\psi_x=\phi_x-i\phi_y\] and then 
$\tilde{u}=\phi_x$, $-\tilde{v}=\phi_y.$ Thus  $f=\nabla \phi$.  
To finish, using Cauchy-Riemann for $\Omega$, we also conclude that
$\tilde{u}=\psi_y$, $ -\tilde{v}=-\psi_x$ and therefore $f$ is a 
hamiltonian vector field with first integral $\psi$. \end{proof}

A basic property of a holomorphic function is the fact that the contour lines of its real and 
imaginary parts define orthogonal families of curves (see, e.g., \cite{Conway}):
\begin{proposition}Let $\Omega=\phi+i\psi$ be a complex potential in a simply connected domain $D$. 
If $z_0\in D$ and $\Omega'(z_0)\neq0$ then $\psi(x,y)=\psi(z_0)$ and $\phi(x,y)=\phi(z_0)$ 
intersect transversally at $z_0$.\end{proposition}
\begin{corollary} Let $\Omega=\phi+i\psi$ and $f=\overline{\Omega'}$ be 
a holomorphic potential
and its complex velocity field associated respectively. Then $\psi$ is a first integral of 
$z'=f(z)$ and $\phi$ is a first integral of $z'=if(z)$.
\end{corollary}
\begin{proof}The integrability of the system $z'=f(z)$ has already been verified 
in the Theorem \ref{teopsi}. 
The integrability of the system $z'=if(z)$ follows immediately from the fact that the field 
$if(z)$ is a 90 degree rotation of the field $f(z)$ 
and from the fact that the level curves of $\phi$ and $\psi$ are orthogonal.\end{proof}

Let  $f=u+iv$ be a complex representation of the velocity field
defined in a simply connected domain of $\C.$
\begin{itemize}
	\item[(a)] We say that  $f$ is  \textit{irrotational} if $\operatorname{rot}(f)=(0,0,v_x-u_y)=(0,0,0)$. In particular, 
	this condition implies that $f$ is a gradient vector field.
	\item[(b)] We say that $f$ is  \textit{incompressible} if  $\operatorname{div}(f)=u_x+u_y=0.$
	\item[(c)] If  $f$ is irrotational and  incompressible then we say that $f$ is  
	\textit{ideal}.
	\end{itemize}	
\begin{theorem} $f$ is the complex representation of the velocity field of an ideal fluid if and only if  $\overline{f}$ is holomorphic.\end{theorem}
\begin{proof}The proof is a direct consequence of the Cauchy-Riemann equations
	\[f=u+iv\quad\mbox{ideal}\Leftrightarrow v_x=u_y,\quad u_x=-v_y\Leftrightarrow \overline{f}=u-iv\quad\mbox{satisfies CR}.\]\end{proof}

    \begin{figure}[h]
	\centering 
	\includegraphics[width=10cm, height=6cm]{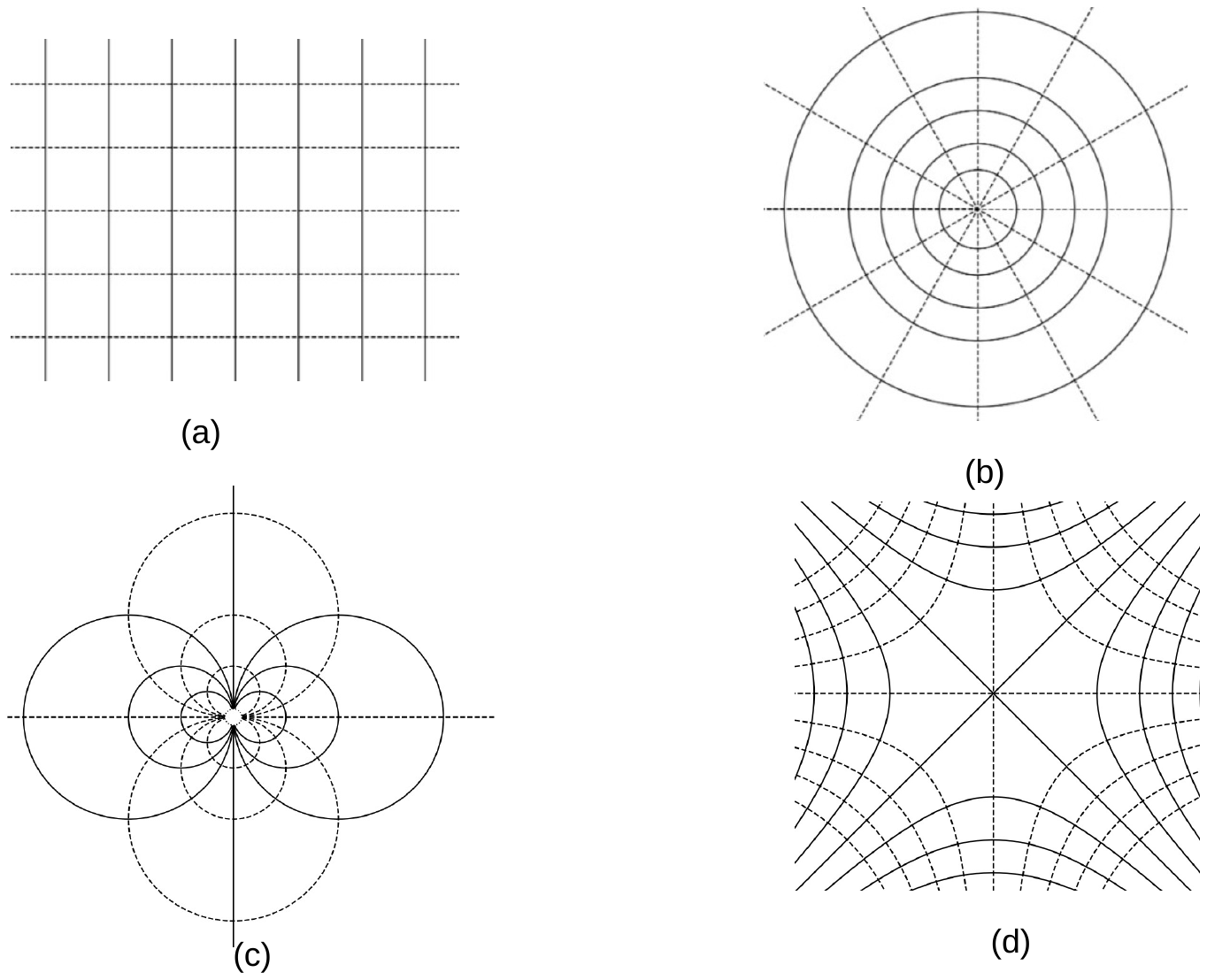} 
	\caption{Equipotential Lines and Streamlines of: (a) $f(z)=1$, (b) $f(z)=z$, (c) $f(z)=z^2$ and (d) $f(z)=\dfrac{1}{z}$.}\label{ffig1}
\end{figure}

Any holomorphic function  $\Omega=\phi+i\psi$ can be interpreted as a 
complex potential for the two-dimensional flow of an ideal fluid. The curves $\phi=constant$ are 
called \textit{equipotential curves} and the curves $\psi=constant$ are called \textit{streamlines}. The complex velocity is 
$f=\overline{\Omega'}$, that is, $\Omega$ is a primitive of $\overline{f}.$ In Figure~\ref{ffig1} we show equipotential lines and streamlines for some elementary 
complex potentials: (a) $f(z)=1$, (b) $f(z)=z$, (c) $f(z)=z^2$, and (d) $f(z)=\dfrac{1}{z}$.

A less trivial example is given by $f(z) = \frac{z^2}{1+z}$, whose equipotential 
and streamlines are depicted in Figure~\ref{ffig2}.

\begin{figure}[h]
	\centering 
	\includegraphics[width=8cm, height=4cm]{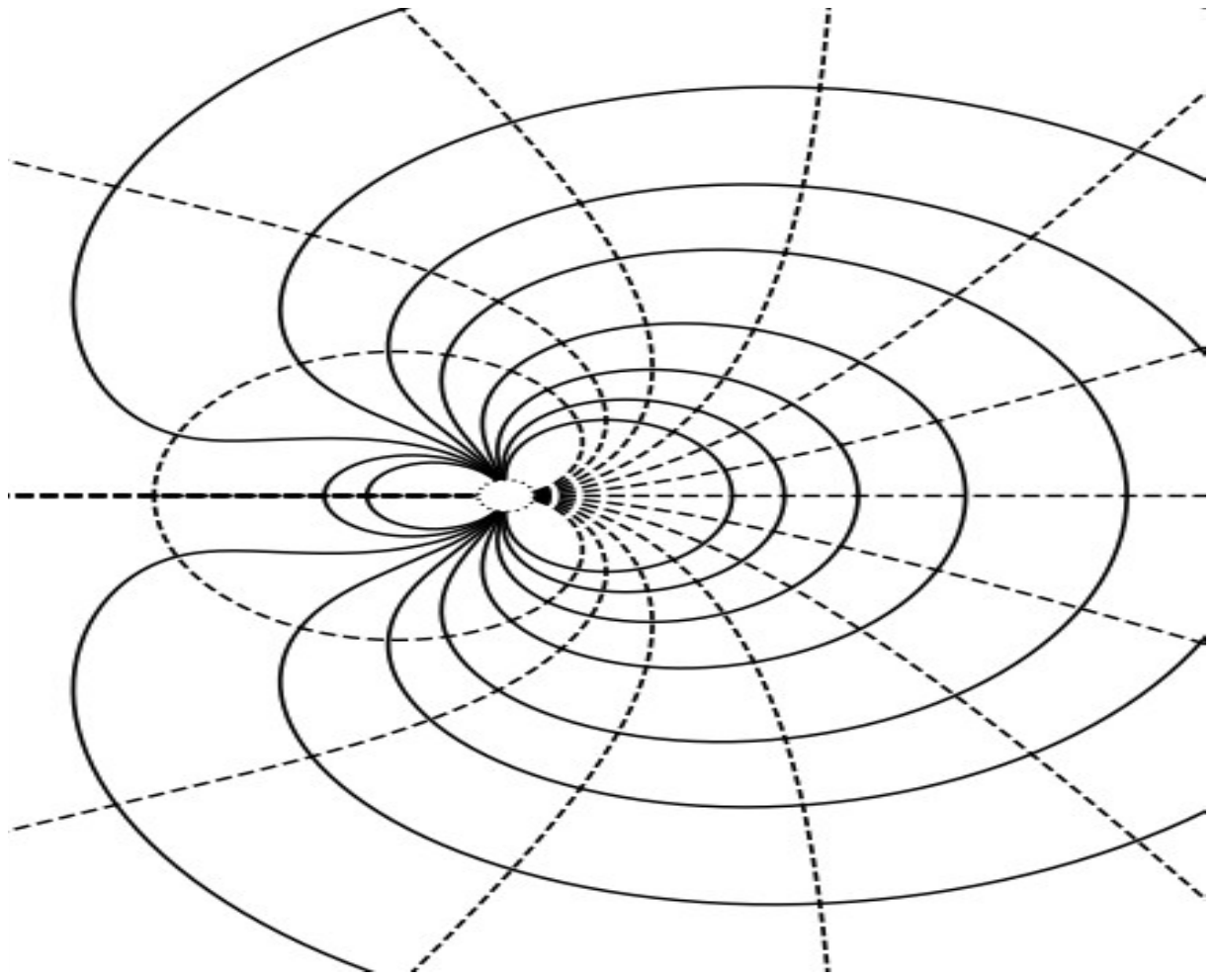} 
	\caption{Equipotential Lines and Streamlines of $f(z)=\dfrac{z^2}{1+z}$.}\label{ffig2}
\end{figure}
            
\begin{example}
Let $k$ and $A$ be nonzero real constants.
\begin{itemize}
\item[(a)]
The vector field
$f(z)=\frac{k}{2\pi}\frac{z-z_0}{|z-z_0|^2}$
is the complex representation of the velocity field of an ideal fluid. Indeed,
$\overline{f(z)}=-\frac{k}{2\pi}\frac{1}{z-z_0}$
is holomorphic in $\C\setminus\{z_0\}$.

\item[(b)]
The vector field $f(z)=k\overline{z}$ also represents an ideal fluid flow, since
$\overline{f(z)}=kz$ is holomorphic in $\C$.

\item[(c)]
For the complex potential $\Omega(z)=\frac{k}{2}z^2$ we have
$f(z)=\overline{\Omega'(z)}=k\overline{z}$.
Writing $z=x+iy$, the streamlines are given by $xy=\text{constant}$.

\item[(d)]
If $\Omega(z)=Az$, then $f(z)=\overline{\Omega'(z)}=A$, and the streamlines are
$y=\text{constant}$.
\end{itemize}
\end{example}

\begin{definition} We say that $f$ is anti-holomorphic if $\overline{f}$ is holomorphic.
\end{definition}
In other words, $f$ is the complex representation of the velocity field of an ideal fluid if 
and only if $f$ is anti-holomorphic.
	
\begin{theorem} \label{teoalg}The following statements about polynomial potentials  and velocity 
fields are valid.
\begin{itemize}
\item[(a)]If $f$ is  an anti-holomorphic polynomial  map in some neighborhood of $z_0\in\C$ 
then  all its trajectories are algebraic.
\item[(b)] If $\Omega$ is a polynomial function, then all trajectories of $f=\overline{\Omega'}$ are algebraic.
\end{itemize} \end{theorem}

\begin{proof} In fact, there exists $k\in\mathbb{N}$
such that $\overline{f(z)}=\sum_0^k a_n(z-z_0)^n$  and thus $\Omega=\int\overline{f}dz$ 
is a polynomial potential and therefore the trajectories of $f$ are algebraic. To prove  (b), it suffices to see that $\psi=\Im(\Omega)$ is a first polynomial integral.\end{proof}

\begin{proposition}Let $f$ be a holomorphic map defined in an open subset $U\subseteq \C$. 
Then the streamlines of $f$ and of $\frac{1}{\overline{f}}$ are the same in $U\setminus\{z:f(z)=0\}.$
\end{proposition}
\begin{proof}This immediately follows from the fact that the velocity fields 
differ by multiplying a positive function $f=|f|^2\frac{1}{\overline{f}}.$	\end{proof}
	
We now make some remarks on the existence of primitives. A domain $U\subseteq\C$ is said to be \textit{star--shaped} with respect to a point $z_0\in U$ if, for every $z\in U$, the straight line segment joining $z_0$ to $z$ is entirely contained in $U$. For instance, the slit plane $\C\setminus(-\infty,0]$ is star--shaped with respect to any point on the positive real axis. A basic property of holomorphic functions $f:U\to\C$ is that if $U$ is star--shaped, then $f$ admits a primitive in $U$.

 \begin{proposition}\label{star}
Let $f$ be a holomorphic function defined on a punctured disk
$D=D(0,R)\setminus\{0\}$, where either $0$ is the single zero of $f$ in $D$, or $0$ is an isolated singularity of $f$ and $f$ has no zeros in $D$. Let $L$ be a ray starting at $0$. Then both $f$ and $1/f$ admit primitives in $U=D\setminus L$.
\end{proposition}
\begin{proof}
The set $U=D\setminus L$ is star--shaped with respect to any point on the opposite side of $L$, and since $1/f$ is holomorphic in $U$, the result follows.
\end{proof}
         
\begin{definition}Let $f$ be a holomorphic map as in Proposition \ref{star}. We say that 
$\Omega=\int\frac{1}{f}dz$ is a potential of $f$ in $U.$
\end{definition}
	
Let us note that the above definition is natural. In fact, as $f$ and of $\frac{1}{\overline{f}}$
have the same streamlines and  $\frac{1}{\overline{f}}$  is the velocity field associated with the potential 
$\Omega(z)=\int\overline{\frac{1}{\overline{f}}}dz=\int {\frac{1}{f}}dz.$

\begin{theorem} \label{teofn}The complex potentials of the velocity fields in normal form are in the  following table.\end{theorem}

\renewcommand{\arraystretch}{1.25}
\begin{center}
\begin{tabular}{l c l}
\toprule
$f(z)$ & Fig. & $\Omega(z)=\phi(x,y)+i\psi(x,y)$ \\
\midrule
\multirow{2}{*}{$z^n,\; n\in\mathbb{Z}$}
& \multirow{2}{*}{\eqref{ffig1}}
& $\displaystyle 
\phi(x,y)=\frac{(x^2+y^2)^{\frac{1-n}{2}}
\cos\!\big((1-n)\arctan\frac{y}{x}\big)}{1-n}$ \\[0.8ex]
& 
& $\displaystyle 
\psi(x,y)=\frac{(x^2+y^2)^{\frac{1-n}{2}}
\sin\!\big((1-n)\arctan\frac{y}{x}\big)}{1-n}$ \\
\midrule
\multirow{2}{*}{$\displaystyle \frac{z^n}{1+z^{n-1}}$}
& \multirow{2}{*}{\eqref{ffig2}}
& $\displaystyle 
\phi(x,y)=\frac{1}{2}\log(x^2+y^2)
+\frac{(x^2+y^2)^{\frac{1-n}{2}}
\cos\!\big((1-n)\arctan\frac{y}{x}\big)}{1-n}$ \\[0.8ex]
&
& $\displaystyle 
\psi(x,y)=\frac{(x^2+y^2)^{\frac{1-n}{2}}
\sin\!\big((1-n)\arctan\frac{y}{x}\big)}{1-n}
+\arctan\frac{y}{x}$ \\
\bottomrule
\end{tabular}
\end{center}
\begin{proof}
The expressions follow from direct computation of
$\Omega(z)=\int \frac{1}{f}\,dz$ and separation into real and imaginary parts.
\end{proof}
In summary, we have that, given a holomorphic function $\Omega =\phi+i\psi$, 
defined in a simply connected domain, the level curves of the imaginary part $\psi =c$, 
are the integral curves of the vector field $g=\overline{\Omega'}$, which are also the integral curves of the 
field $f=\frac{1}{\Omega'}$ in the region where $\Omega'(z)\neq0$.\\

%

\section{Circulation, Net Flow and Complex Time}\label{sec:net}

Complex potentials provide a local description of the dynamics of the system
$z' = f(z)$ whenever a primitive of $1/f$ exists. In such cases, the imaginary part
of the primitive plays the role of a stream function, whose level sets the trajectories of the flow. This description, however, strongly depends on the
existence of a single-valued primitive and therefore may fail in multiply connected
domains or in the presence of singularities.

Circulation and net flow offer a natural extension of this framework. By considering
line integrals of the conjugate field $\overline{f}$ along closed curves, these
quantities capture global features of the flow that are invisible to local potentials.

\subsection{Circulation and net flow}\label{sec:net1}

Complex potentials provide a local description of the dynamics of the system
$z' = f(z)$ whenever a primitive of $1/f$ exists. In such cases, the imaginary part
of the primitive plays the role of a stream function, whose level sets the trajectories of the flow. This description, however, strongly depends on the
existence of a single-valued primitive and therefore may fail in multiply connected
domains or in the presence of singularities.

Circulation and net flow offer a natural extension of this framework. By considering
line integrals of the conjugate field $\overline{f}$ along closed curves, these
quantities capture global features of the flow that are invisible to local potentials.

\textit{Circulation} refers to the following question: how much the 
flow tends to flow around a curve $C$. The \textit{net flow} is the difference between the entry rate 
and the exit rate of the fluid in the region limited by $C$. Intuitively, if the flow rotates counterclockwise 
around the $C$ curve, we will have positive circulation. If the flow enters the region surrounded by $C$ 
the circulation will be zero and the net flow will be negative  and if the flow leaves the region surrounded
 by $C$ the circulation will be zero and the net flow will be positive (for example a source).

Let us rigorously define these concepts. The \textit{circulation} of $f=u+iv$ is defined by
$\int_CfTds=\int_C udx+vdy.$ The \textit{net flow} is defined by
$\int_C fN ds=\int_C udy-vdx.$
We can combine the two formulas into a single expression
\[\int_C \overline{f}dz=[\mbox{circulation}]+\textit{i}.[\mbox{net flow}].\]	

\begin{example}  Consider the following examples.
\begin{itemize}
\item[(a)] Taking $C$ as the curve $|z|=1$. $f(z)=(z-1)^2$ generates a 
movement with negative circulation and zero net flow as it does not cross $C$. 
In fact, $\int_C \overline{f} dz=-4\pi.$
\item[(b)] Let us take $f(z)=(1+i)z$ and $C$ given by the counterclockwise curve $|z|=1$.
We have $\int_C \overline{f} dz=2\pi +2\pi i.$
\item[(c)] Let us take $f(z)=\overline{\cos z}$ and $C$ given by the counterclockwise square of vertices $1,i,-1,-i$.
We have $\int_C \overline{f} dz=\int_C\cos z dz= 0,$ by Cauchy's theorem.
\item[(d)] Let us take $f(z)=\frac{k}{\overline{z}-\overline{z_0}}$ and $C$ given by the counterclockwise curve around $z_0$. We have, denoting $k=a-bi,$ and using the  Cauchy's integral formula $\int_C \overline{f} dz=2\pi b+2\pi a i.$
\end{itemize}
\end{example}

\begin{example}  \normalfont{We can use circulation and net flow to analyze the flow of linear systems. 
Let $a,b\in\R$ be constant and $f(z)=(a+ib)z$ be a linear velocity field. 
Let $C_{\e}$ be the curve parameterized by $C_{\e}(t)=\e e^{it}$ where $\e>0$ and $t\in[0,2\pi]$.
Thus, we have that 
\[\int_{C_{\e}} \overline{f}dz=
\int_{C_{\e}}(a-bi)\overline{z}dz=\e(b+ai).\]
Then, we conclude that the sign of $a$ determines whether the equilibrium point at the origin is attracting
 ($a<0$) or a repelling ($a>0$). Furthermore, the sign of $b$ determines the direction of rotation. 
If $b>0$ the direction is counterclockwise and if $b<0$ the direction is clockwise.}
\end{example}

\subsection{Holomorphic Differential Equations with Complex Time}\label{sec:net2}
Denote $z=x+iy$, $ f=u+iv$,  $T=t+is$,
and consider the differential equation 
\[z'(T)=f(z(T)).\]
Assume that $f$ is holomorphic in an open subset $U\subseteq\C$ and $z(T)$ is a holomorphic solution 
also defined in an open subset $V\subseteq\C$.
Due to the Cauchy Riemann equations we must have
\[x_t=y_s,\quad x_s=-y_t.\]
Since $z'=u+iv=x_t+iy_t=y_s-x_si$ the equation corresponds to two systems
\begin{equation}\Re:  x_t=u,\quad y_t=v; \quad
 \Im:  x_s=-v ,\quad y_s=u.\label{tempocomplexo}\end{equation}

 Geometrically, these two systems describe flows along the two orthogonal families of curves associated with the complex potential $\Omega$: system $\Re$ flows along streamlines ($\psi=$ constant), while system $\Im$ flows along equipotential lines ($\phi=$ constant). This explains why, in Figure~\ref{ffig1}, the families of curves correspond to trajectories with real and pure imaginary time, respectively.
 \begin{proposition}Let $\Omega=\phi+i\psi$ be the complex potential of  $f=u+iv$, 
 defined in a simply connected domain $D$. 
 Then $\psi$ is a first integral
 of system \eqref{tempocomplexo}-$\Re$ and $\phi$ is a first integral of system \eqref{tempocomplexo}-$\Im$.
 \end{proposition}
\begin{proof}
It follows directly from the fact that $\Omega=\int\frac{1}{f}dz$ and so 
\[\Omega'=\phi_x+i\psi_x=\dfrac{u}{u^2+v^2}-i\dfrac{v}{u^2+v^2}=\psi_y-i\phi_y\]
and therefore
\[\Re: \left\{ \begin{array}{l} x_t=(u^2+v^2)\psi_y \\ y_t=-(u^2+v^2)\psi_x\end{array},\right. \quad
 \Im: \left\{ \begin{array}{l} x_s=-(u^2+v^2)\phi_y\\ y_s=(u^2+v^2)\phi_x\end{array}\right. .\]
 As $\psi$ and $\phi$ are first integrals of the systems $\Re/(u^2+v^2)$ and $\Im/(u^2+v^2)$, 
 then they are also of \eqref{tempocomplexo}-$\Re$ and \eqref{tempocomplexo}-$\Im$, respectively. 
\end{proof}

\begin{corollary}Let $\Omega=\phi+i\psi$ be the complex potential of  $f=u+iv$, 
defined in a simply connected domain $D$. 
Then $\psi$ is a first integral of $z'=f$ and $\phi$ is a first integral of $z'=if.$
\end{corollary}
\begin{proof}This follows directly from the fact that 
\[f=(u,v)=(u^2+v^2)(\psi_y,-\psi_x)\] and \[if=(-v,u)=(u^2+v^2)(-\phi_y,\phi_x).\]
\end{proof}

An alternative proof of the above proposition can be done using the variable separation method. 
\[\dfrac{dz}{dT}=f\implies \int\dfrac{dz}{f}=\int dT\implies \phi+i\psi=t+si+a+bi,\]
where $a,b$ are real constants. Deriving the last relation with respect to $t$ we get
\[(\phi_x,\phi_y)(x_t,y_t)+i(\psi_x,\psi_y)(x_t,y_t)=1\]
and thus 
\[(\psi_x,\psi_y)(x_t,y_t)=0.\] It means that $\psi(x,y)$ is a first integral of system \eqref{tempocomplexo}-$\Re$.
Similarly, differentiating both sides with respect to $s$ we conclude that $\phi(x,y)$ is a first integral of \eqref{tempocomplexo}-$\Im$.\\

\begin{example} \normalfont{Let us consider the differential equation $z'=1$ satisfying the initial condition $z(0)=x_0+iy_0=z_0.$ We have $u=1,v=0$ and  $\Re:  x_t=1, \quad y_t=0;$  
 $\Im:  x_s=0,\quad y_s=1.$ Thus
$x=t+c(s),$ $y=d(s)$ with $x(0)=x_0=c(0),$ $y(0)=y_0=d(0).$
We have $d'=y_s=1,$ $c'=x_s=-y_t=0,$
and then $c=x_0,$ $d=s+y_0.$
We conclude that $x=x_0+t,y=y_0+s,$ 
\[z=(x_0+t)+i(y_0+s)=(t+si)+(x_0+iy_0)=T+z_0.\] The phase portrait of this constant flow is shown in Figure~\ref{ffig1}(a).}
\end{example}
Of course, if we use the variable separation method we found the same result
\[\dfrac{dz}{dT}=1\implies z=T+K,\quad z_0=z(0)=K\implies z=T+z_0.\]

\begin{example} \normalfont{Let us consider the differential equation $z'=z$ satisfying the initial condition 
$z(0)=x_0+iy_0=z_0.$ We have
$u=x,v=y$ and $\Re:  x_t=x , \quad y_t=y;$ 
 $\Im: x_s=-y,\quad  y_s=x.$
 Thus $x=c(s)e^t,$  $y=d(s)e^t$
with $x(0)=x_0=c(0), y(0,0)=y_0=d(0).$ Proceeding as in the previous example, we easily obtain
$z=z_0e^T.$  Of course, if we use the variable separation method we found the same result
\[\dfrac{dz}{dT}=z\implies z=Ke^T,\quad z_0=z(0)=K\implies z=z_0e^T.\] The phase portrait of this linear flow is shown in Figure~\ref{ffig1}(b).}
\end{example}

\begin{example}  \normalfont{Let us consider the differential equation $z'=z^2$ satisfying the initial condition $z(0)=x_0+iy_0=z_0.$
We have $u=x^2-y^2$, $v=2xy$ and 
$\Re:x_t=x^2-y^2, \quad y_t=2xy;$ 
 $\Im:  x_s=-2xy,\quad y_s=x^2-y^2 .$
 In this case, using separation of variables we obtain
$z=\frac{-1}{T-\frac{1}{z_0}}.$
The phase portrait of this quadratic flow, showing both real time trajectories (tangent to the $x$-axis) and pure imaginary time trajectories (tangent to the $y$-axis), is displayed in Figure~\ref{ffig1}(c).}
 \end{example}

\begin{example}  \normalfont{Let us consider the differential equation $z'=\frac{1}{z}$ satisfying the initial condition $z(0)=x_0+iy_0=z_0.$ In this case, using separation of variables we obtain
$\frac{z(T)^2}{2}=T+\frac{z_0^2}{2}.$ The phase portrait of this reciprocal flow is shown in Figure~\ref{ffig1}(d).}
 \end{example}
 These four basic examples illustrate the geometric structure of holomorphic flows with complex time. For more complex nonlinear cases, such as $z'=\frac{z^2}{1+z}$, the relationship between real and imaginary time flows follows the same principles but with richer geometric structure, as shown in Figure~\ref{ffig2}.
 
\section{Monic and Centered Cubic Polynomial Differential Equation}\label{sec:cubic}

We begin this section by recalling that a polynomial \( p(z) = A_n z^n + \dots + A_0 \), $A_k=a_k+ib_k$,  is said to be 
\emph{monic} and \emph{centered} if \( A_n = 1 \) and \( A_{n-1} = 0 \).

Any polynomial of degree greater than or equal to two can be transformed into a monic centered polynomial 
via an affine change of variables \( z \rightarrow \alpha z + \beta \) with \( \alpha \neq 0 \). 
This procedure reduces the number of parameters and simplifies the classification. 
The vanishing of the coefficient \( A_{n-1} \) means that the sum of the zeros of the 
polynomial is zero, according to Viète's Theorem. Geometrically, this implies that the 
barycenter of the roots is located at the origin.

The set of monic centered polynomials is closed under rotations \( z \rightarrow e^{i\theta} z \), 
which preserve the leading term and keep \( A_{n-1} = 0 \). 
Moreover, we can normalize the poles at infinity, and in doing so, we index the separatrices. 
The coefficient $A_n$ determines the rate at which the trajectories approach infinity.  
On the other hand, when the term $A_{n-1}$ vanishes, it ensures that we can index the separatrices without ambiguity.

We will classify the possible phase portraits of monic and centered polynomial systems of degree 3.  

The phase portrait of a monic cubic polynomial system on the Poincaré disk 
is characterized by a finite number of canonical regions. 
There are three types of such regions:
\begin{enumerate}
    \item \textbf{Center-type region:} 
    the phase portrait in the finite part of the plane is organized around a single center.
    
    \item \textbf{Sepal-type region:} 
    this region contains a unique finite equilibrium point, and every finite orbit in the region 
    has both its $\alpha$- and $\omega$-limit sets equal to that point.
    
    \item \textbf{$\alpha$--$\omega$ region:} 
    this region contains two equilibrium points, and every finite orbit 
    has its $\alpha$-limit at one of them and its $\omega$-limit at the other.
\end{enumerate}

\begin{theorem}\label{teosepala}
Given a monic and centered cubic polynomial system of degree $3$, 
the global phase portrait on the Poincaré disk can be decomposed 
into $2$, $3$, or $4$ canonical regions. 
Moreover, it holds that there can only exist 
$1$ or $3$ center-type regions, 
$1$ or $2$ $\alpha$--$\omega$ --type regions, 
and $2$ or $4$ sepal-type regions. The possible configurations are as follows:
\begin{enumerate}
    \item three center-type regions;
    \item one center-type region and one $\alpha$--$\omega$ region;
    \item one center-type region and two sepal-type region;
    \item four sepal-type regions;
    \item two sepal-type regions and one $\alpha$--$\omega$ region;
    \item two $\alpha$--$\omega$ regions.
\end{enumerate}
\end{theorem}

\subsection{Proof of Theorem \ref{teosepala}}

The first step will be to determine the dynamics at infinity.
Consider $\dot{z} =z^3+A_ 1z+A_0=u + i\,v$, where 
\[
u = p_3+ \big(a_1p_1-b_1q_1\big)+a_0,
\qquad
v = q_3+\big(b_1 p_1 + a_1 q_1\big) + b_0,
\] and $p_k=\Re (x+iy)^k$ and $q_k=\Im (x+iy)^k$.

The resulting planar system is given by
\[
\dot{x} = x^{3} - 3x y^{2} + \ldots,
\qquad
\dot{y} = 3x^{2} y - y^{3} + \ldots
\]

To study the equilibrium points at infinity, we use the Poincaré compactification.  
The expression of the Poincaré compactification in the chart $U_1$ is given by
\[
\dot{s} = w^{3}\left( -s\,u\left(\frac{1}{w},\frac{s}{w}\right) + v\left(\frac{s}{w}\right) \right),
\qquad
\dot{w} = -w^{4} u\left(\frac{1}{w},\frac{s}{w}\right).
\]
The expression in $U_2$ is given by
\[
\dot{s} = w^{3}\left( u\left(\frac{s}{w},\frac{1}{w}\right) - s\,v\left(\frac{s}{w},\frac{1}{w}\right) \right),
\qquad
\dot{w} = -w^{4} v\left(\frac{s}{w},\frac{1}{w}\right).
\]
Note that, in these charts, a point $(s,w)$ at infinity has coordinates $(s,0)$.  
For the chart $U_1$, we must therefore study the system
\[
\dot{s} = 2s(s^{2} + 1), \qquad \dot{w} = 0.
\]
Note that $s = 0$ is a saddle of the system above.  
For the other charts, following the same steps, we also find a saddle point.  
We denote the four points at infinity by $e_{0}$, $e_{1}$, $e_{2}$, and $e_{3}$.

The next step in obtaining the phase portrait is to classify the equilibrium points.  
According to the Euler–Jacobi formula, if $z_{1}$, $z_{2}$, and $z_{3}$ are simple zeros of $p(z)$, we have that
\[\frac{1}{p'(z_1)}+\frac{1}{p'(z_2)}+\frac{1}{p'(z_3)}=0.\]
As a consequence of this formula, if two equilibria of $z' = p(z)$ are centers, then the third one must also be a center. 
Similarly, if two equilibria are of node type, then the third one will also be a node.  
Another consequence is that if one of the equilibria is not of center type, then there must 
necessarily exist two equilibria with independent stabilities. In summary, the possible configurations for 
the equilibrium points of a cubic polynomial system are the following: (a) 3 centers, (b) 3 nodes, (c) 1 triple, (d) 3 foci,
(e) 1 center and 1 double, (f) 1 node and 1 double, (g) 1 focus and 1 double, (h) 1 center and 2 foci, 
(i) 1 node and 2 foci and (j) 1 center, 1 node and 1 focus.

In what follows we denote by $S_{ij}$ the separatrix connecting the equilibria $e_{i}$ and $e_{j}$.  
Moreover, in order  to visualize the phase portraits we use rectifying coordinates.  
More precisely, if $\Phi(z)$ is a primitive of $\frac{1}{p(z)}$, then the change of coordinates $w = \Phi(z)$ satisfies  
$\dot{w} = \Phi'(z) \cdot \dot{z} = \frac{1}{p(z)} \cdot p(z) = 1$.

In the figures below, the left panel represents the phase portrait on the Poincaré disk, while the right panel shows its rectification via the change of coordinates $w=\Phi(z)$, for which the flow becomes horizontal straight lines. Finite orbits correspond to horizontal lines in the $w$--plane, whose ends represent their $\alpha$ and $\omega$--limits. In particular, an orbit may come from infinity in negative time and return to infinity in positive time through a separatrix $S_{ij}$, or connect finite equilibria depending on the configuration.

\subsubsection{(a) 3 centers.} The phase portrait, in the Poincaré disk, is composed of three regions.  
One of the regions is bounded by the separatrix $S_{10}$ and the arc $(e_{0},e_{1}) \subset S^{1}$,  
another is bounded by the two separatrices $S_{10}$ and $S_{32}$,  
and the third is bounded by the separatrix $S_{32}$ and the arc $(e_{2},e_{3}) \subset S^{1}$. 
An example of a polynomial system that exhibits this configuration is  $\dot{z} = z^{3} - i z.$ In this case we have three center-type regions. The phase portrait is sketched in Figure~\ref{ret1}.
\begin{figure}[H]
	\includegraphics[width=8cm, height=5cm]{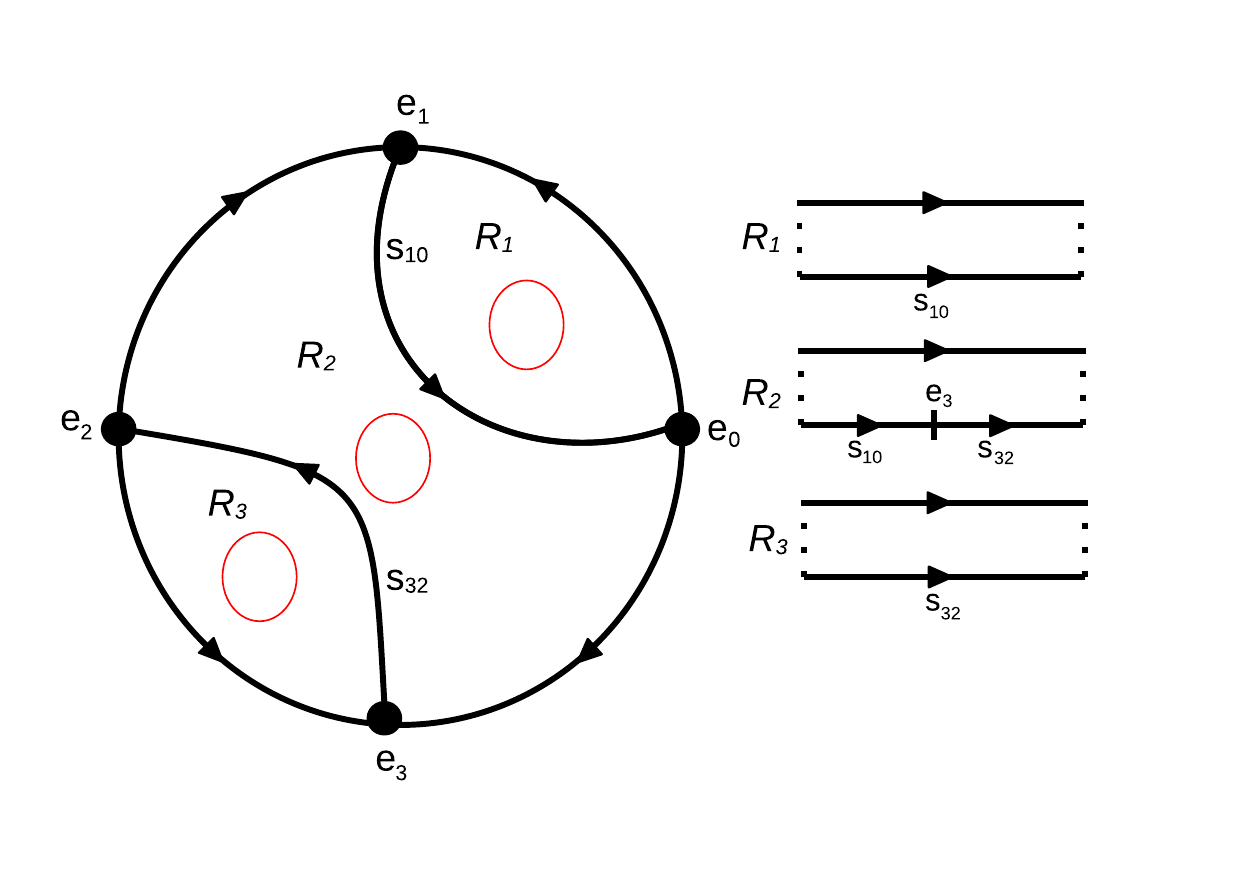} 
	\caption{\small {On the left, the phase portrait in the Poincaré disk of $\dot{z} = z^3 - iz$; 
	on the right, the rectification of the phase portrait via the coordinate change $w = \Phi(z)$.}}\label{ret1}
\end{figure}

\subsubsection{(b) 3 nodes.} The phase portrait, in the Poincaré disk, is composed of two regions.   
One of the regions is bounded by the separatrices $ S_1$ and $S_3$ and by the arcs $(e_0,e_1)$ and $(e_0,e_3)$;  
another is the complementar one. 
An example of a polynomial system that exhibits this configuration is  $\dot{z} = z^{3} - z.$
In this case we have two $\alpha$--$\omega$ regions.
The phase portrait is sketched in Figure~\ref{ret2}.
\begin{figure}[H]
	\includegraphics[width=8cm, height=5cm]{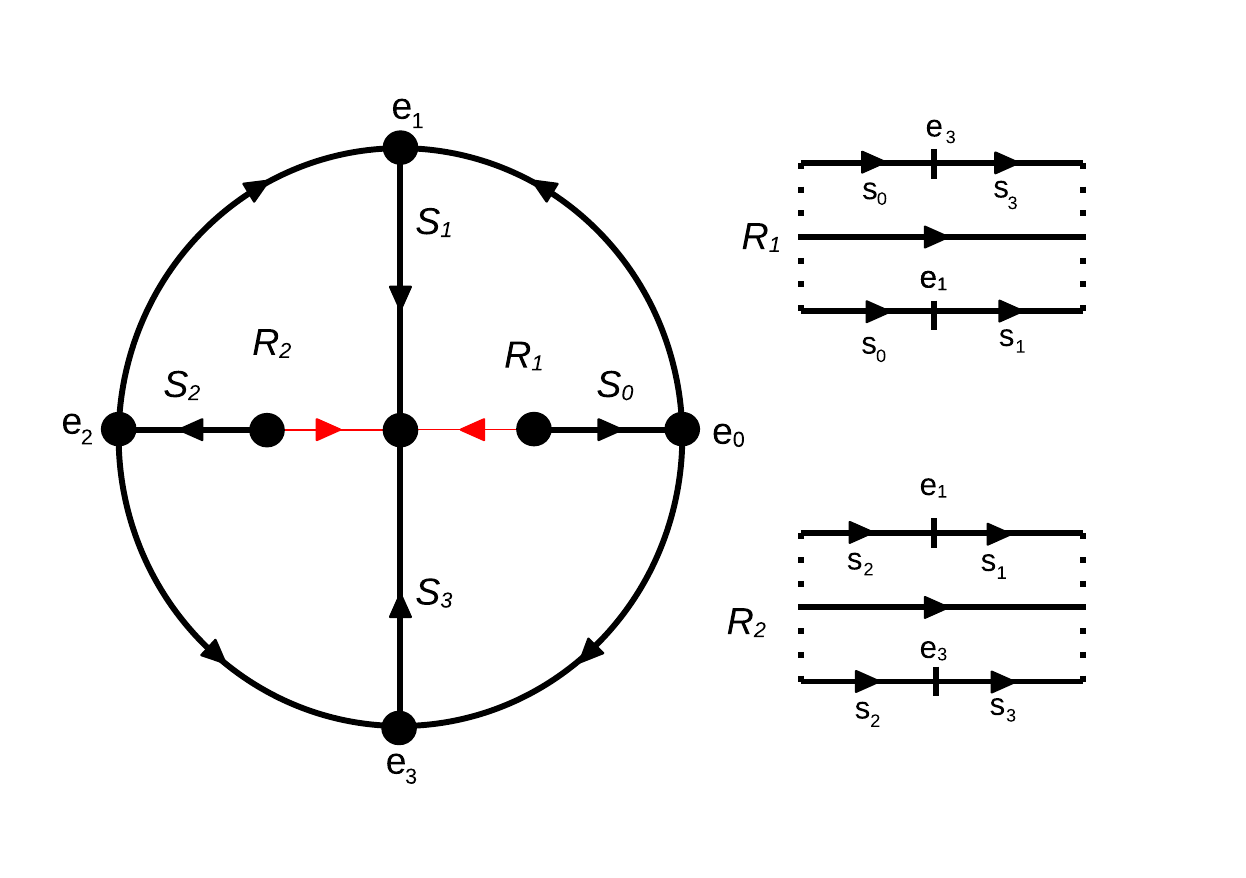} 
	\caption{\small{On the left, the phase portrait in the Poincaré disk of $\dot{z} = z^3 - z$; 
	on the right, the rectification of the phase portrait via the coordinate change $w = \Phi(z)$.}}\label{ret2}
\end{figure}

\subsubsection{(c) 1  triple.} The phase portrait, in the Poincaré disk, is composed of four regions (each quadrant is a region).  
An example of a polynomial system that exhibits this configuration is  $\dot{z} = z^{3}.$
In this case we have four sepal-type regions.
The phase portrait is sketched in Figure~\ref{ret3}.
\begin{figure}[H]
	\includegraphics[width=8cm, height=5cm]{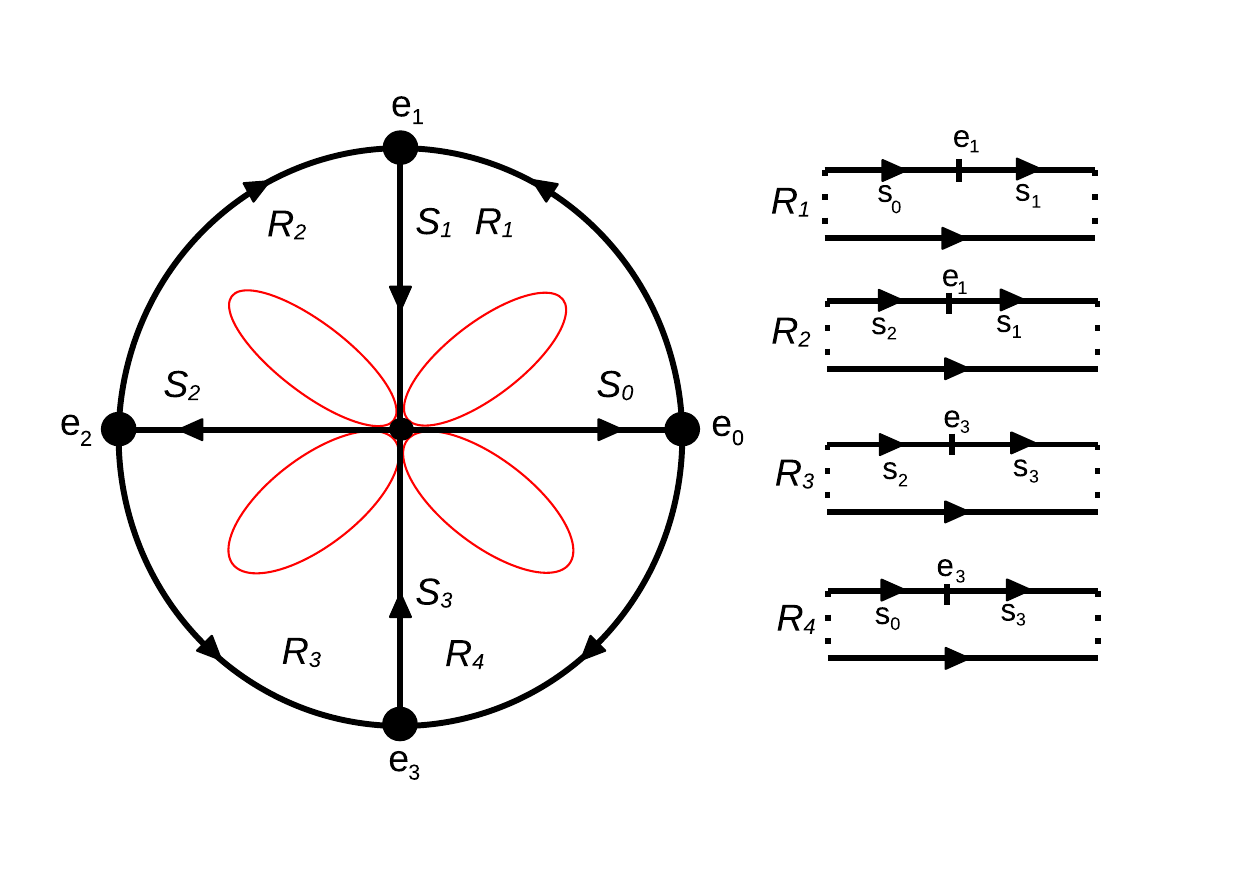} 
	\caption{\small{On the left, the phase portrait in the Poincaré disk of $\dot{z} = z^3$; 
	on the right, the rectification of the phase portrait via the coordinate change $w = \Phi(z)$.}}\label{ret3}
\end{figure}

\subsubsection{(d) 1 attracting focus, and 2 repelling foci.} The phase portrait, in the Poincaré disk, is composed of two regions.
The regions share as a common boundary the union of the separatrices $S_1$ and $S_3$. 
An example of a polynomial system that exhibits this configuration is  $\dot{z} = z^3+(4+6i)z+(4-12i)$. 
In this case we have two $\alpha$--$\omega$ regions.
The phase portrait is sketched in Figure~\ref{ret5}.
\begin{figure}[H]
	\includegraphics[width=8cm, height=5cm]{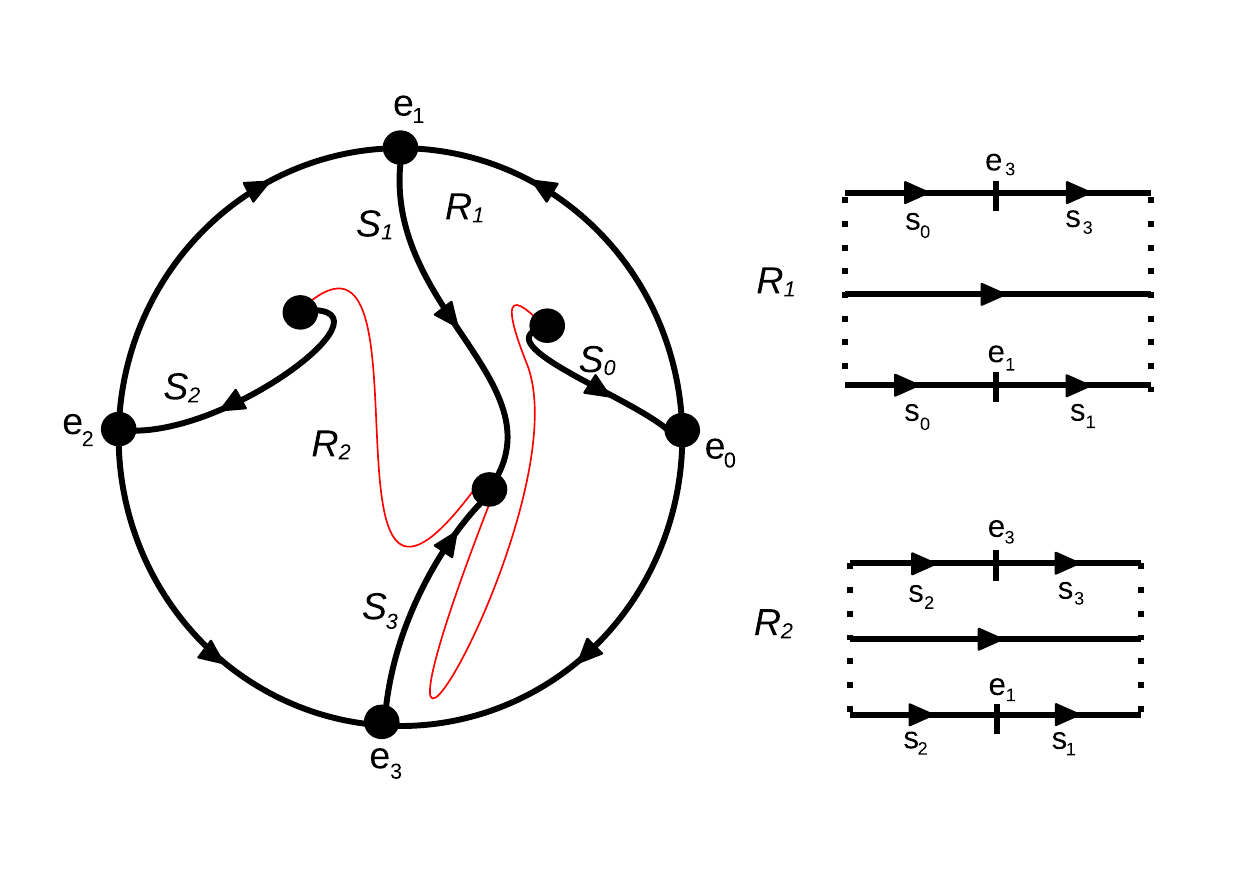} 
	\caption{\small{On the left, the phase portrait in the Poincaré disk of $\dot{z} = z^3+(4+6i)z+(4-12i)$; 
	on the right, the rectification of the phase portrait via the coordinate change $w = \Phi(z)$.}}\label{ret5}
\end{figure}

\subsubsection{(e) 1 center and 1 double.} The phase portrait, in the Poincaré disk, is composed of three regions.
One region is  bounded by the separatrices $S_{1},S_0$ and by the arc $(e_0,e_1)$, another is  bounded by 
the separatrix $S_{23}$ and by the arc $(e_2,e_3)$ and the third is the complementar one.
An example of a polynomial system that exhibits this configuration is  $\dot{z} = z^3-3iz+\sqrt{2}(-1+i)$. 
In this case we have one center-type region and two sepal-type region.
The phase portrait is sketched in Figure~\ref{ret6}.
\begin{figure}[H]
	\includegraphics[width=8cm, height=5cm]{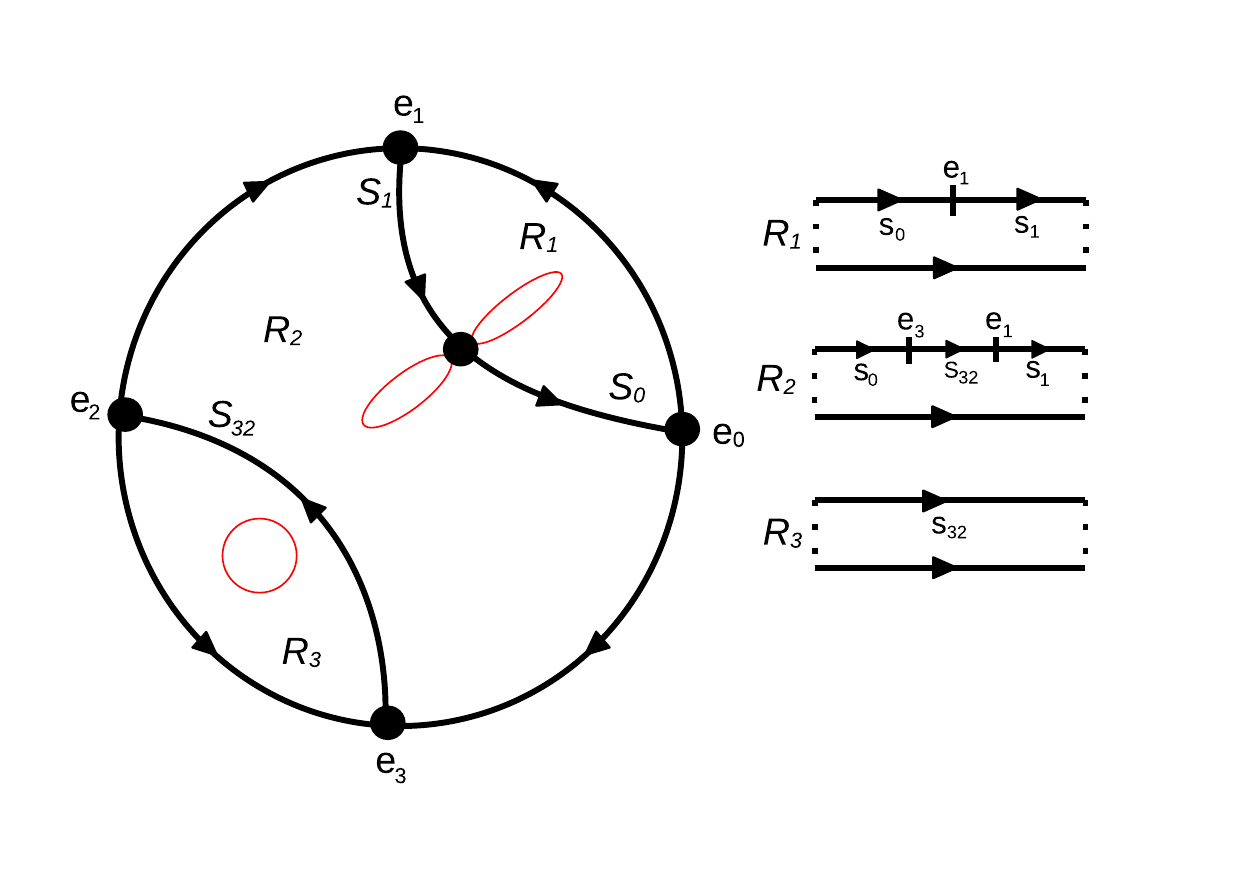} 
	\caption{\small{On the left, the phase portrait in the Poincaré disk of $\dot{z} = z^3-3iz+\sqrt{2}(-1+i)$; 
	on the right, the rectification of the phase portrait via the coordinate change $w = \Phi(z)$.}}\label{ret6}
\end{figure}

\subsubsection{(f) 1 repelling node and 1 double.} The phase portrait, in the Poincaré disk, is composed of three regions.
One region is  bounded by the separatrices $S_{1}$, $S_0$ and by the arc $(e_0,e_1)$, another is  bounded by 
the separatrices $S_{3}$, $S_0$ and by the arc and by the arc $(e_0,e_3)$ and the third is the complementar one.
An example of a polynomial system that exhibits this configuration is  $\dot{z} = z^3-3z+2$. 
In this case we have two sepal-type regions and one $\alpha$--$\omega$ region.
The phase portrait is sketched in Figure~\ref{ret7}.
\begin{figure}[H]
	\includegraphics[width=8cm, height=5cm]{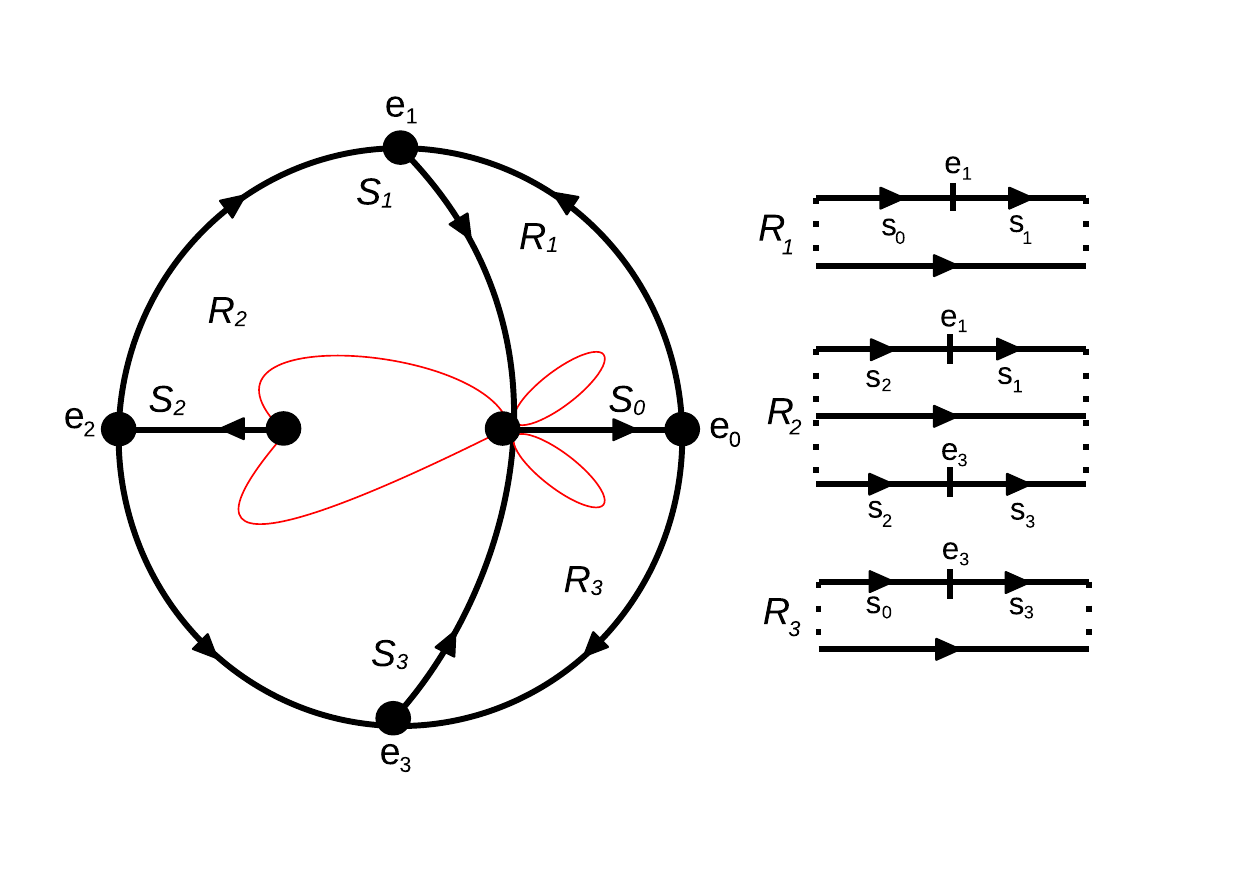} 
	\caption{\small{On the left, the phase portrait in the Poincaré disk of $\dot{z} = z^3-3z+2$; 
	on the right, the rectification of the phase portrait via the coordinate change $w = \Phi(z)$.}}\label{ret7}
\end{figure}

\subsubsection{(g) 1 attracting focus and 1 double.} The phase portrait, in the Poincaré disk, is composed of three regions.
One region is  bounded by the separatrices $S_{1}$, $S_0$ and by the arc $(e_0,e_1)$,  another is  bounded by 
separatrices $S_{1}$, $S_2$ and by the arc $(e_2,e_1)$  and the third is the complementar one.
An example of a polynomial system that exhibits this configuration is  $\dot{z} = z^3+(9-12i)z-22-4i$. 
In this case we have two sepal-type regions and one $\alpha$--$\omega$ region.
The phase portrait is sketched in Figure~\ref{ret8}.
\begin{figure}[H]
	\includegraphics[width=8cm, height=5cm]{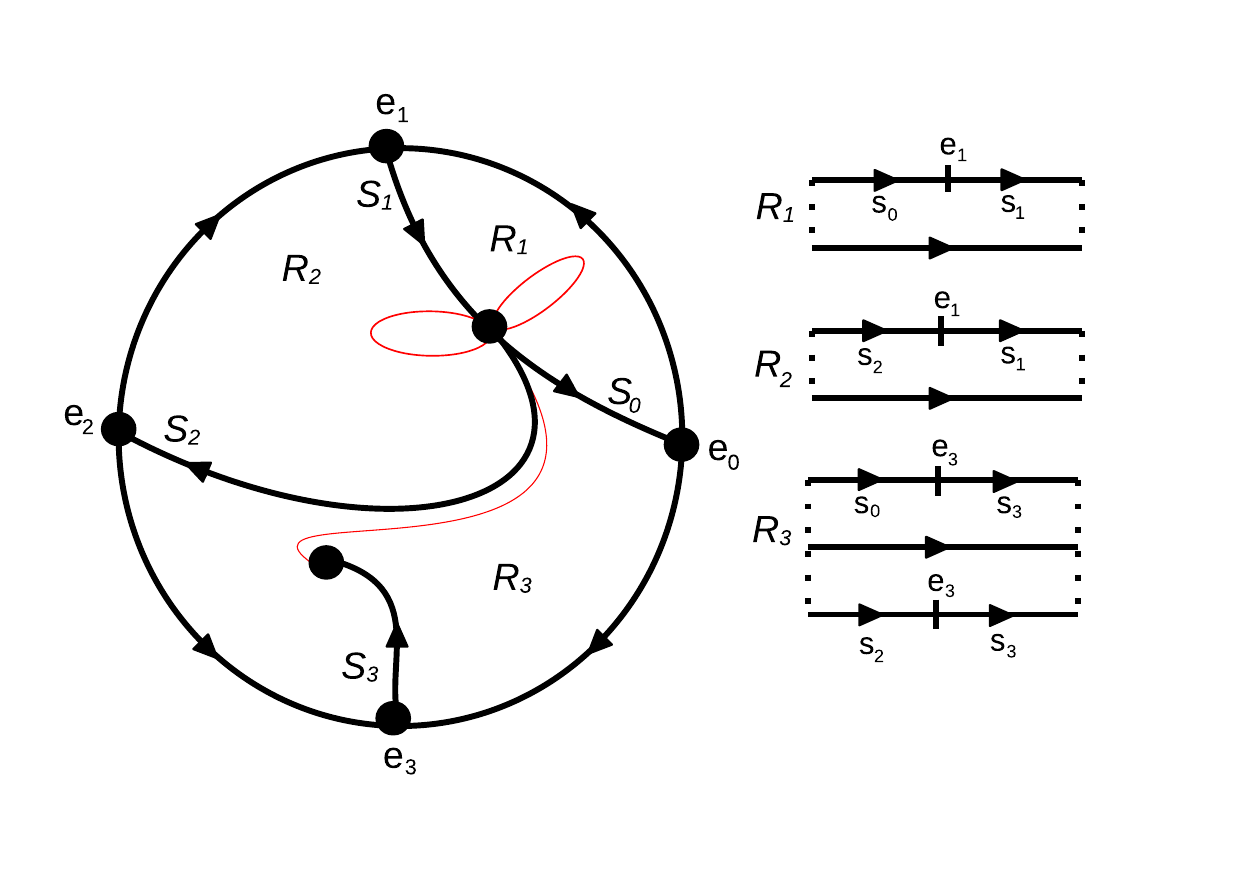} 
	\caption{\small{On the left, the phase portrait in the Poincaré disk of $\dot{z} = z^3+(9-12i)z-22-4i$; 
	on the right, the rectification of the phase portrait via the coordinate change $w = \Phi(z)$.}}\label{ret8}
\end{figure}

\subsubsection{(h) 1 center, 1 attracting focus, and 1 repelling focus.} The phase portrait, in the Poincaré disk, is composed of two regions.  
One of the regions is bounded by the separatrix $S_{10}$ and by the arc $(e_0,e_1)$. 
The other region is the complement of this one in the disk.
An example of a polynomial system that exhibits this configuration is  $\dot{z} = z^3+3iz+(5-5i)$. 
In this case we have one center-type region and one $\alpha$--$\omega$ region. The phase portrait is sketched in Figure~\ref{ret4}.
\begin{figure}[H]
	\includegraphics[width=8cm, height=5cm]{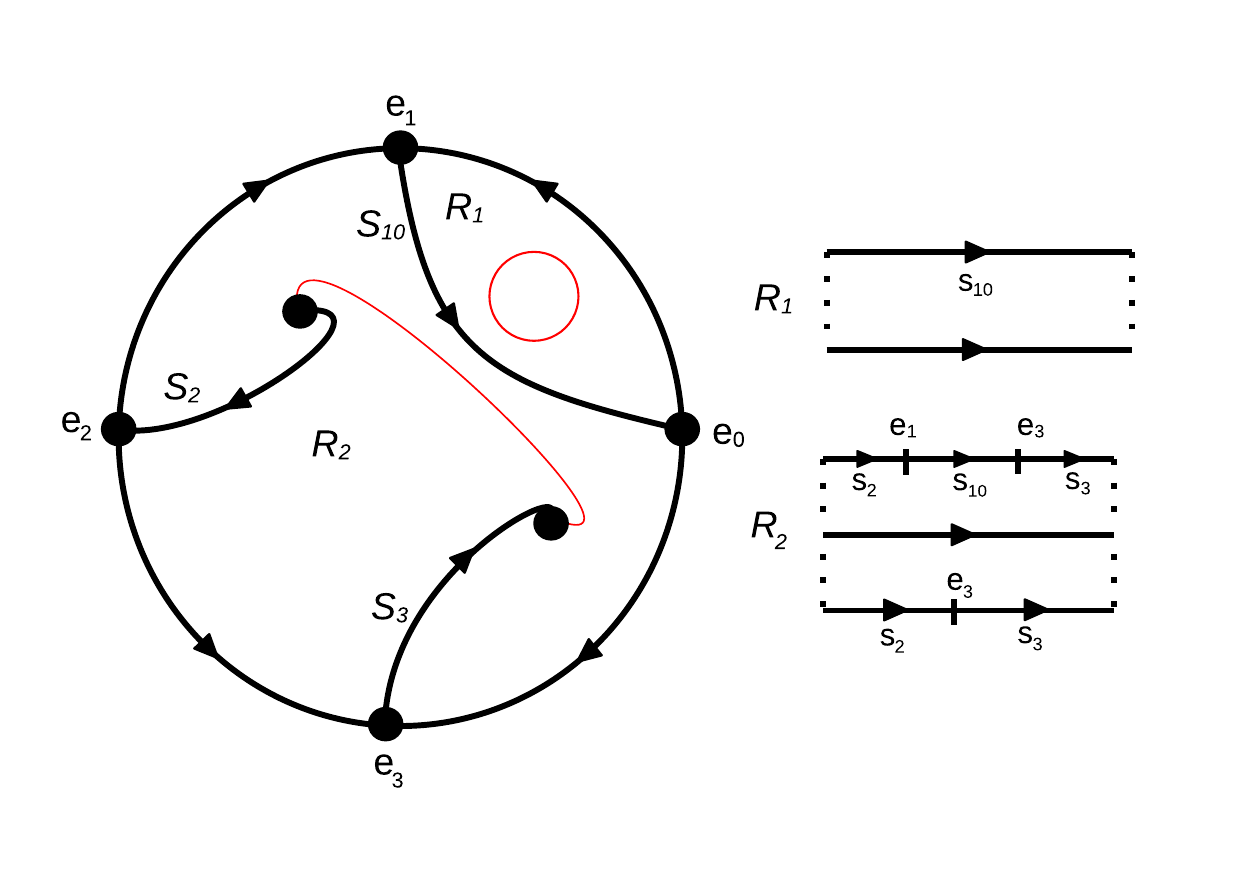} 
	\caption{\small{On the left, the phase portrait in the Poincaré disk of $\dot{z} = z^3+3iz+(5-5i)$; 
	on the right, the rectification of the phase portrait via the coordinate change $w = \Phi(z)$.}}\label{ret4}
\end{figure}

\subsubsection{(i) 2 attracting focus and 1 repelling node.} The phase portrait, in the Poincaré disk, is composed of two regions.
One region is  bounded by the separatrices $S_{2}$, $S_0$ and by the arc $(e_0,e_2)$,  another is  the complementar one.
An example of a polynomial system that exhibits this configuration is  $\dot{z} = z^3+\frac{z}{4}-\frac{5}{4}$. 
In this case we have two $\alpha$--$\omega$ regions.
The phase portrait is sketched in Figure~\ref{ret9}.
\begin{figure}[H]
	\includegraphics[width=8cm, height=5cm]{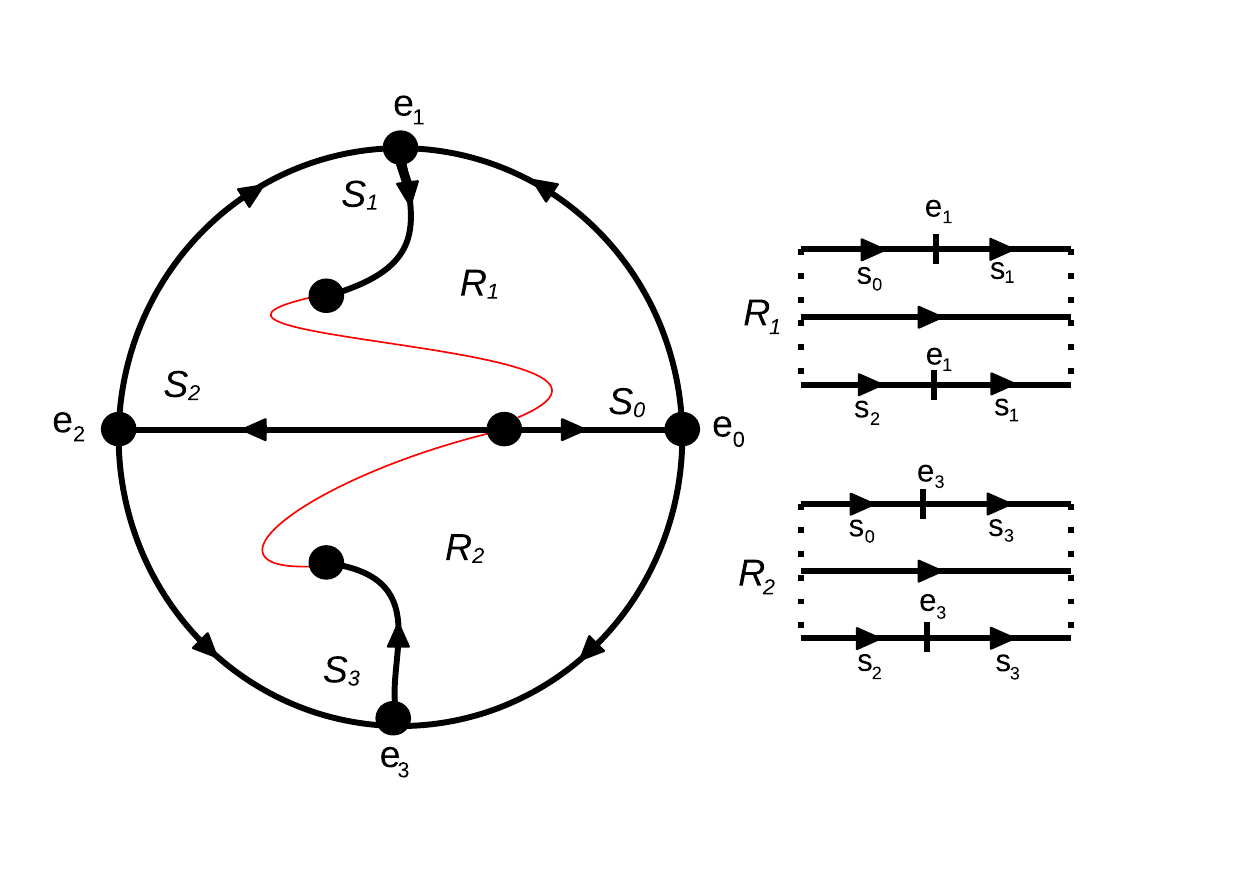} 
	\caption{\small{On the left, the phase portrait in the Poincaré disk of $\dot{z} = z^3+\frac{z}{4}-\frac{5}{4}$; 
	on the right, the rectification of the phase portrait via the coordinate change $w = \Phi(z)$.}}\label{ret9}
\end{figure}

\subsubsection{(j) 1 center, 1 attracting focus and 1 repelling node.} The phase portrait, in the Poincaré disk, is composed of two regions.
One region is  bounded by the separatrix $S_{12}$ and by the arc $(e_2,e_1)$,  another is  the complementar one.
An example of a polynomial system that exhibits this configuration is  $\dot{z} = z^3 + (-343/7500 + 1152/625 i) z + (-1542592/1687500 - 1433619/1687500 i)$. 
In this case we have one center-type region and one $\alpha$--$\omega$ region.
The phase portrait is sketched in Figure~\ref{ret10}.
\begin{figure}[H]
	\includegraphics[width=8cm, height=5cm]{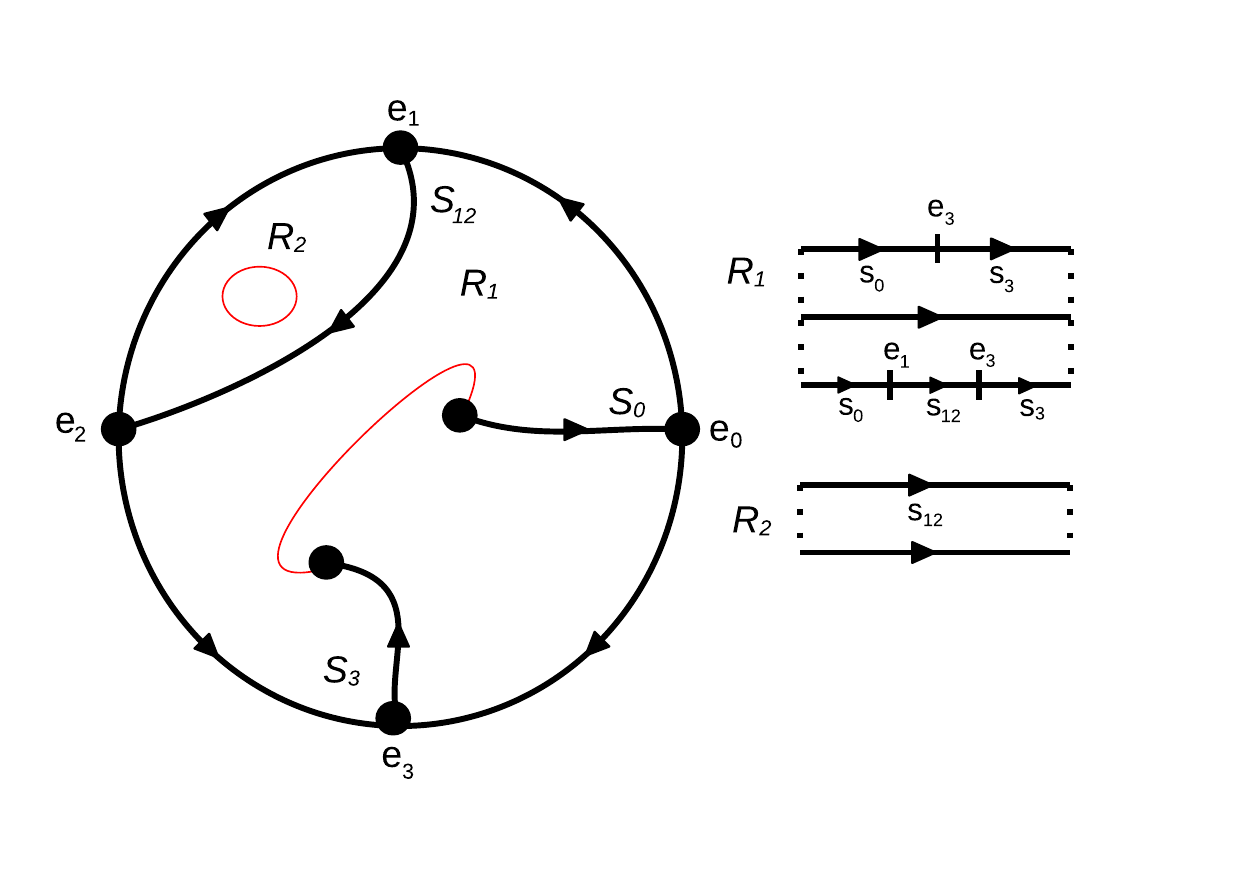} 
	\caption{\small{On the left, the phase portrait in the Poincaré disk of  $\dot{z} = z^3 + (-343/7500 + 1152/625 i) z + (-1542592/1687500 - 1433619/1687500 i)$; 
	on the right, the rectification of the phase portrait via the coordinate change $w = \Phi(z)$.}}\label{ret10}
\end{figure}

\subsection{Bernoulli Equations.}
 Let us consider Bernoulli equations with constant coefficients
 \[z'+\alpha z=\beta z^n,\quad \alpha,\beta\in\C\setminus{0}, \quad n=2,3,....\]
 As before $z=x+iy$, $T=t+is$ and the initial condition is $z(0)=z_0$. We carry out the usual variable change
 $w=z^{1-n}$ and thus the equation obtained is
 \[w'+\alpha(1-n)w=(1-n)\beta, \quad w(0)=w_0=z_0^{1-n}\] whose solution is given by
 \[w=\dfrac{\beta}{\alpha}\left( 1+(w_0-1)e^{-(1-n)\alpha T}\right).\]
 In the initial variable we obtain
 \[z=\dfrac{\beta}{\alpha}\left( 1+(z_0^{1-n}-1)e^{-(1-n)\alpha T}\right)^{\dfrac{1}{1-n}}.\]
 
 \begin{proposition} The real system \eqref{tempocomplexo}-$\Re$ of the Bernoulli equations 
 $z'+\alpha z=\beta z^n,$ with  $n=2,3,....$ has $n$ equilibrium point $0,z_1,...,z_{n-1}$ where
 $z_k^{n-1}=\dfrac{\alpha}{\beta}$. Moreover we have
 \begin{itemize}
 \item If $\Re\alpha>0$ (resp. $<0$) and $\Im\alpha=0$ then $0$ is an attracting (resp. repelling) node and $z_k$, $k=1,...,n-1$ are repelling (resp. attracting) nodes.
 \item If $\Re\alpha>0$ (resp. $<0$) and $\Im\alpha\neq0$ then $0$ is an attracting  (resp. repelling) focus and $z_k$, $k=1,...,n-1$ are repelling (resp. attracting) focus.
 \item If $\Re\alpha=0$ and $\Im\alpha\neq0$ then $0$ is a center and $z_k$, $k=2,...,n-1$ are
 centers too.
 \end{itemize}
 \end{proposition}
 \begin{proof} The equilibrium points are the solutions of 
 $f(z)=-\alpha z+\beta z^n=0$ which are $z=0$ and the and the roots of unity $z^{n-1}=\dfrac{\alpha}{\beta}$.
 Since $f'(0)=-\alpha$ and $f'(z_k)=-\alpha+n\beta \dfrac{\alpha}{\beta}=(n-1)\alpha$ the 
 classification of the equilibrium depends only of the real and imaginary parts of $\alpha$.
 \end{proof}


Let us now consider the particular case where $\beta=1$, that is, the Bernoulli system is monic and centered.

\begin{proposition}
Consider the Bernoulli equation
\[
z' = z^n - \alpha z,
\]
where $\alpha = a + bi$ with $a,b \in \mathbb{R}$.  
The global phase portrait of this system on the Poincaré disk is described as follows:
\begin{itemize}
\item[(a)] There exist $n$ finite equilibria, one of which is the origin, while the remaining $n-1$ lie on the circle of radius $|\alpha|^{\frac{1}{n-1}}$.  
If $a > 0$, the origin is an attractor and the other equilibria are repellers.  
If $a < 0$, the origin is a repeller and the others are attractors.  
If $a = 0$, all finite equilibria are centers.

\item[(b)] There exist $2(n-1)$ equilibria at infinity, denoted by $e_0, e_1, \ldots, e_{2n-3}$, whose angular positions are given by
\[
e_k = \frac{k\pi}{\,n-1\,}, \qquad k = 0, \ldots, 2n-3.
\]
All these equilibria are of saddle type.

\item[(c)] If $a \neq 0$, the global phase portrait consists of $n-1$ canonical regions.  
When $a > 0$, each region is of $\alpha\omega$–type, with $\alpha$–limit on one of the equilibria $z_1, \ldots, z_{n-1}$ and $\omega$–limit at the origin.  
When $a < 0$, the roles are reversed: the $\alpha$–limit is the origin and the $\omega$–limit is one of the equilibria $z_1, \ldots, z_{n-1}$.  
The boundary of each region is formed by separatrices connecting finite equilibria to equilibria at infinity.  
If $a = 0$, the global phase portrait consists of $n$ canonical regions, all of center type, bounded by separatrices connecting the equilibria at infinity.
\end{itemize}
\end{proposition}
\begin{proof}
The proof of (a)  has already been established in the previous proposition.
For (b),  we begin by denoting
\[z^n-\alpha z = u(x,y)+i\,v(x,y), \qquad z=x+iy,\ \ \alpha=a+bi.\]
Equivalently,
\[u(x,y)=\Re\!\big((x+iy)^n\big) - (ax - by),
\qquad
v(x,y)=\Im\!\big((x+iy)^n\big) - (ay + bx).\]
To study the infinite equilibrium points, we use the Poincaré Compactification. The
expressions of Poincaré Compactification in the charts $U_1$ and $U_2$ are given by
\[(U_1)\quad s'=w^n\left(-su\left(\dfrac{1}{w},\dfrac{s}{w}\right)+v\left(\dfrac{1}{w},\dfrac{s}{w}\right)\right),\quad  
w'=-w^{n+1}u\left(\dfrac{1}{w},\dfrac{s}{w}\right)\]
and
\[(U_2)\quad s'=w^n\left(-sv\left(\dfrac{s}{w},\dfrac{1}{w}\right)+u\left(\dfrac{s}{w},\dfrac{1}{w}\right)\right),\quad  
w'=-w^{n+1}v\left(\dfrac{s}{w},\dfrac{1}{w}\right).\]
Performing the computations in chart $(U_1)$ (with the change of variables
$x=\frac{1}{w}$, $y=\frac{s}{w}$ so that $z=\frac{1+is}{w}$), write
\[
(1+is)^n = A_n(s)+i\,B_n(s).
\]
Then
\[
u\!\left(\tfrac{1}{w},\tfrac{s}{w}\right)=\frac{A_n(s)}{w^{n}}-\frac{a-bs}{w},
\qquad
v\!\left(\tfrac{1}{w},\tfrac{s}{w}\right)=\frac{B_n(s)}{w^{n}}-\frac{as+b}{w}.
\]
and 
\[s' = -sA_n(s)+B_n(s)-b(s^2+1)w^{n-1},\qquad
w' =-wA_n(s)+w^{n}\big(a-bs\big).\]
In particular, on the boundary $w=0$ we obtain
\[
s' = -sA_n(s)+B_n(s),\qquad w'=0.
\]

For reference, the polynomials $A_n,B_n$ are
\[
A_n(s)=\sum_{m=0}^{\lfloor n/2\rfloor}\binom{n}{2m}(-1)^m s^{2m}=1+...,
\quad
B_n(s)=\sum_{m=0}^{\lfloor (n-1)/2\rfloor}\binom{n}{2m+1}(-1)^m s^{2m+1}=ns+....
\]
The equilibria of this system are given by
\[
p_k = \tan\!\left(\frac{k\pi}{\,n-1\,}\right).
\]
If $n$ is even, then there are $(n-1)$ equilibria, while if $n$ is odd, there are $(n-2)$ equilibria. 
By performing a similar analysis in chart $(U_2)$, we obtain $n-2$ equilibria given by $q_k=\cot\!\left(\frac{k\pi}{\,n-1\,}\right)$. 
Combining the information from charts $(U_1)$ and $(U_2)$, we obtain $2(n-1)$ equilibria at infinity, denoted by $e_0, e_1, \ldots, e_{2n-3}$, where 
\[e_k = \frac{k\pi}{\,n-1\,},  k = 0, \ldots, 2n-3.\]
To verify that the equilibria are of saddle type, it suffices to compute the determinant of the Jacobian matrix and check that it is negative. 
At the points $p_k$, we have 
\[
\det = -(n-1)\,(1+p_k^2)^{\,n-1},
\]
which satisfies the required condition. Moreover, we can also verify that the signs of $s'$ in the intervals $(p_k, p_{k+1})$ alternate, 
being positive when $k$ is even and negative when $k$ is odd. 

The statements concerning the canonical regions follow from the preliminary analysis we have carried out of the finite and infinite equilibria. 
Consider, for instance, the case $n = 3$ and let us take $\alpha = 1$ to simplify the computations. 
In this case, we have three equilibria ($0$, $1$, and $-1$) in the finite part and four equilibria at infinity. 
The points $e_0$ and $e_2$ are repelling in the angular direction and attracting in the radial direction, 
while the points $e_1$ and $e_3$ are attracting in the angular direction and repelling in the radial direction. 
Thus, the separatrices $S_1$ and $S_3$ must necessarily have their $\omega$-limit at the origin, the unique attracting equilibrium. 
On the other hand, the separatrices $S_0$ and $S_2$ must have their $\alpha$-limit at some repelling equilibrium. 
The only possible configuration is the one shown in the Figure \ref{ret-b-1}. In Figures~\ref{ret-b-1} and \ref{ret-b-2} we show the possible phase portraits for $n=3$, the first one corresponding to the case $\operatorname{Re}(\alpha)>0$ and the second one to the case $\operatorname{Re}(\alpha)=0$.  
Similarly, in Figures~\ref{ret-b-3} and \ref{ret-b-4} we present the possible phase portraits for $n=4$, the first one corresponding to the case $\operatorname{Re}(\alpha)>0$ and the second one to the case $\operatorname{Re}(\alpha)=0$.
\end{proof}
\begin{figure}[ht]
\centering
\begin{minipage}[t]{0.5\textwidth}
\centering
\includegraphics[width=\linewidth]{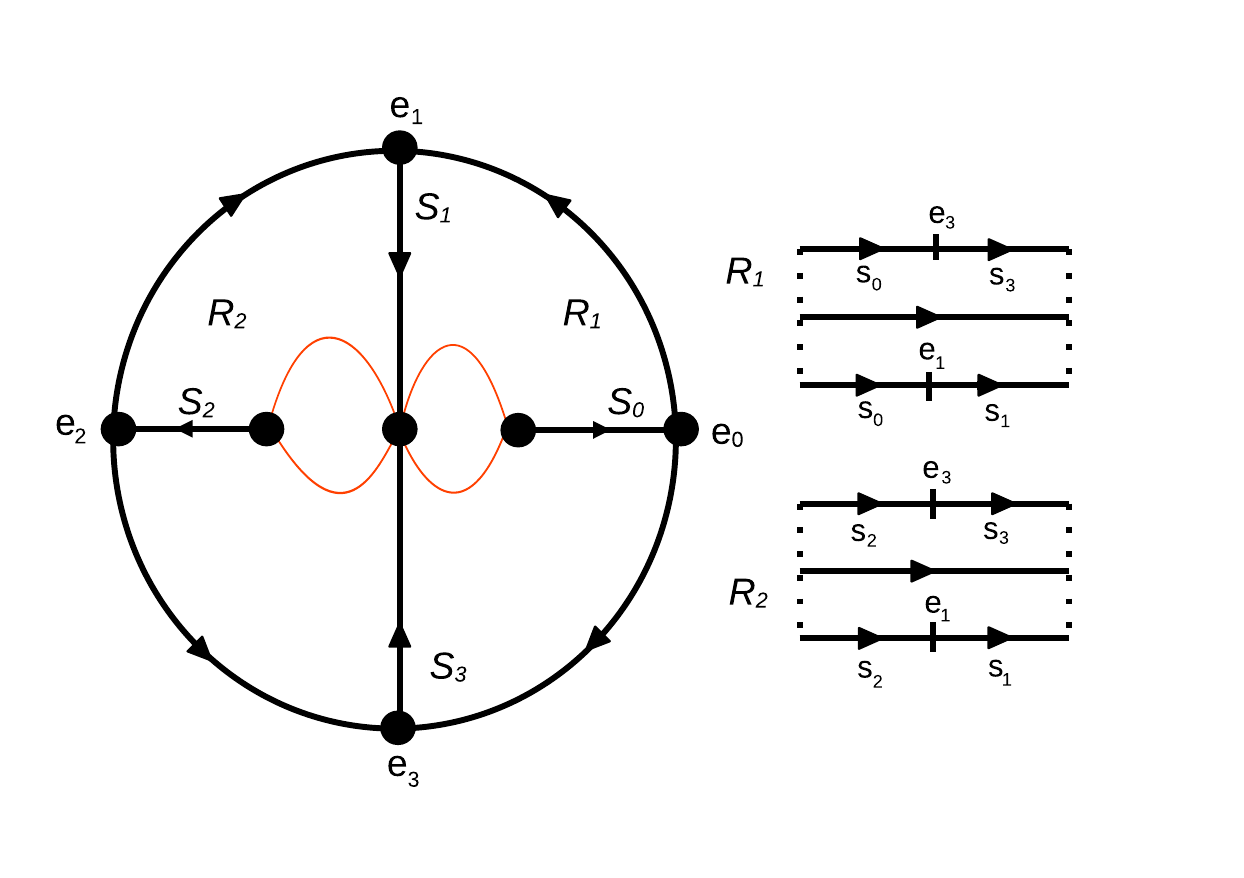} 
\caption{Phase portrait of $\dot{z} = z^3-z$ (left) and its rectification via $w = \Phi(z)$ (right).}
\label{ret-b-1}
\end{minipage}\hfill
\begin{minipage}[t]{0.5\textwidth}
\centering
\includegraphics[width=\linewidth]{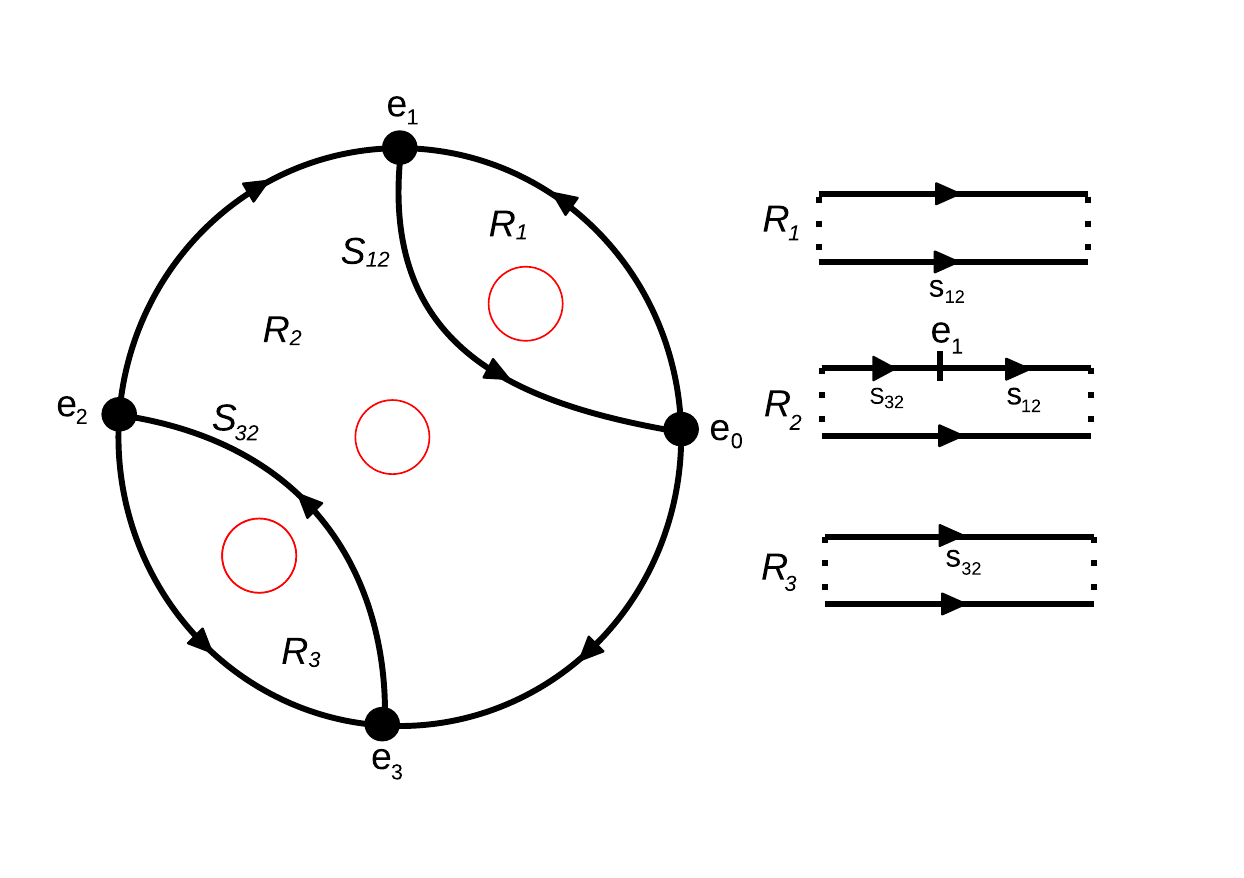} 
\caption{Phase portrait of $\dot{z} = z^3-iz$ (left) and its rectification via $w = \Phi(z)$ (right).}
\label{ret-b-2}
\end{minipage}
\end{figure}

\begin{figure}[ht]
\centering
\begin{minipage}[t]{0.5\textwidth}
\centering
\includegraphics[width=\linewidth]{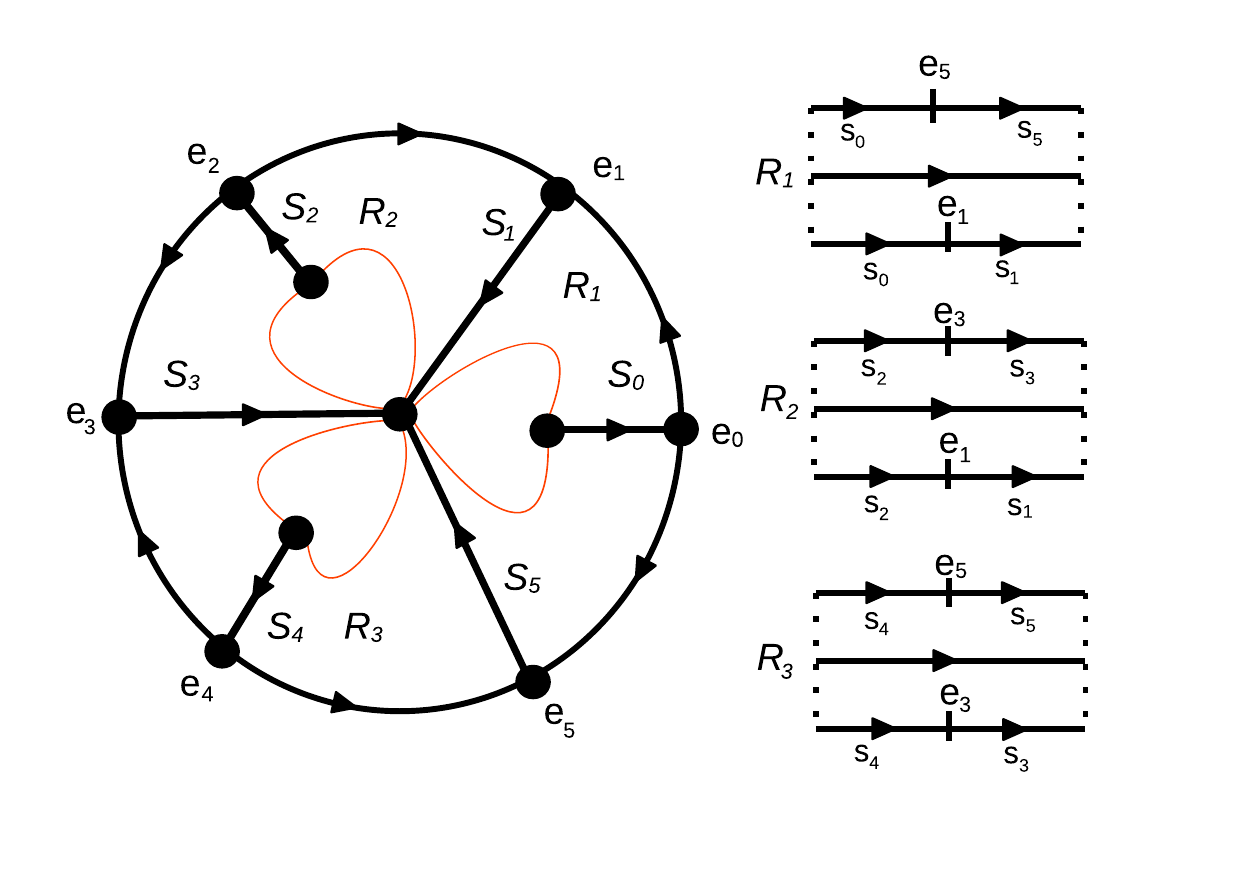} 
\caption{Phase portrait of $\dot{z} = z^4-z$ (left) and its rectification via $w = \Phi(z)$ (right).}
\label{ret-b-3}
\end{minipage}\hfill
\begin{minipage}[t]{0.5\textwidth}
\centering
\includegraphics[width=\linewidth]{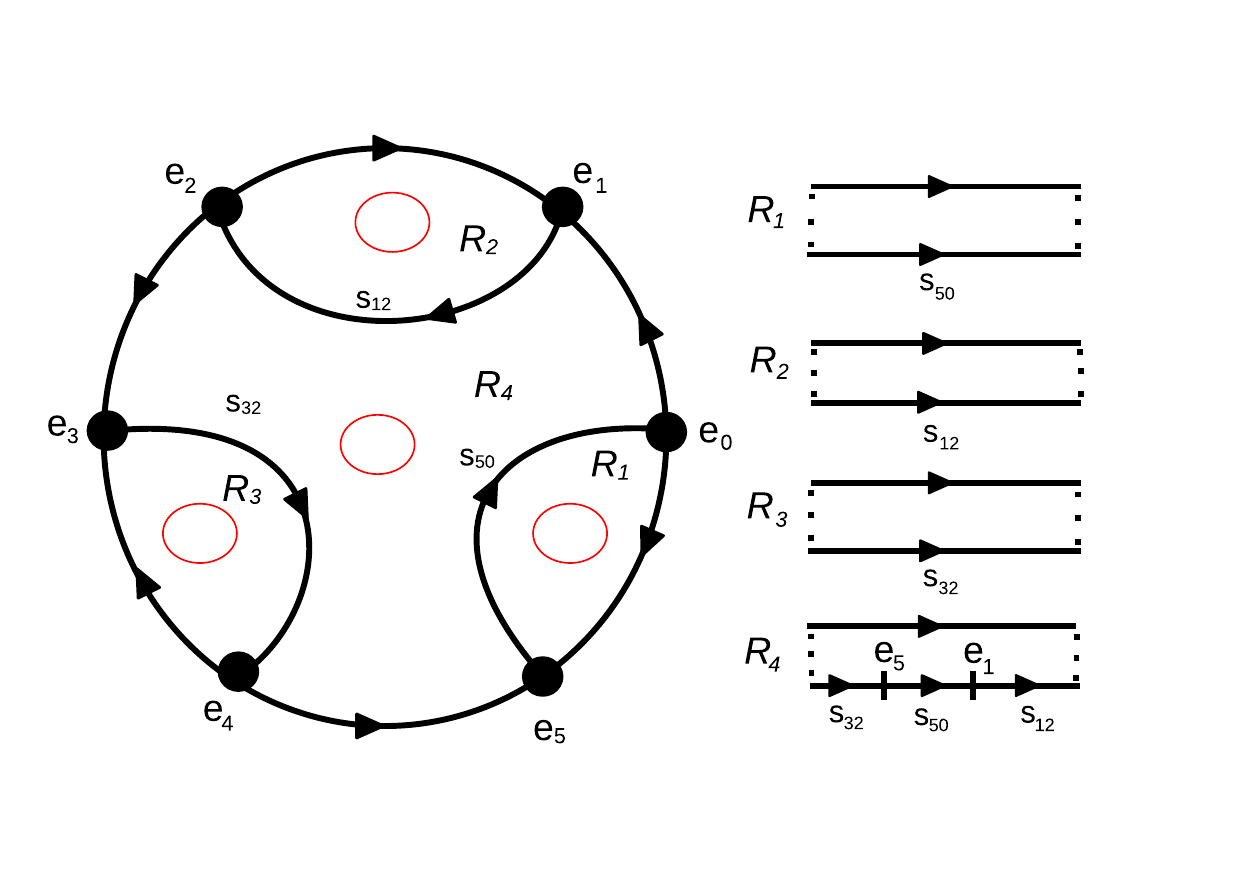} 
\caption{Phase portrait of $\dot{z} = z^4+iz$ (left) and its rectification via $w = \Phi(z)$ (right).}
\label{ret-b-4}
\end{minipage}
\end{figure}

\section{Bounds for limit cycles in a special class of piecewise smooth systems}\label{sec:cycles}
Holomorphic planar vector fields are Hamiltonian with respect to a harmonic potential, consequently, they cannot possess limit cycles. This fundamental obstruction disappears when we consider piecewise holomorphic systems (PWHS), where the loss of global analyticity across a switching manifold opens the door to isolated periodic orbits. Recent works have systematically explored this possibility, establishing lower bounds for the number of limit cycles in piecewise polynomial holomorphic systems \cite{GSR2, GSR3}, providing conditions for their existence, stability, and uniqueness, and even exhibiting explicit piecewise linear examples with up to three limit cycles, a notable improvement over earlier constructions that yielded at most one.

The analysis of PWHS is greatly simplified by the underlying complex structure. The normal forms associated with holomorphic functions admit polar representations whose symmetry, together with the invariance of certain rays, facilitates the explicit construction of first return maps. Moreover, for anti‑holomorphic vector fields of the form $\dot{z} = \overline{p(z)}$ with $p$ polynomial, integration yields a polynomial complex potential $\Omega = \int p(z) dz$, whose imaginary part provides a polynomial first integral. This algebraic property is what makes the matching conditions for closed orbits across the discontinuity line tractable: it reduces the problem to finding common zeros of two explicit symmetric polynomials.

Throughout this section, we fix the switching line as the real axis
\[
\Sigma = \{z \in \mathbb{C} : \Im(z) = 0\},
\]
and adopt the Filippov convention to define the flow on $\Sigma$ (see \cite{Filippov88}).

We exploit this framework to study two families where the polynomial first integral yields particularly sharp upper bounds:
\begin{enumerate}
\item \textbf{Mixed anti‑holomorphic–holomorphic systems}, where one half‑plane is governed by the conjugate of a linear polynomial and the other by a linear holomorphic field. The position of the equilibrium (on or off the switching line) drastically influences the maximum number of limit cycles.
\item \textbf{Piecewise anti‑holomorphic polynomial systems}, where both half‑planes are ruled by the conjugate of polynomials of the same degree, and therefore both possess polynomial first integrals. The bound depends directly on the degree of the polynomials.
\end{enumerate}
The results presented here complement the aforementioned literature by providing \emph{upper bounds} obtained through this unified algebraic approach, thereby delineating the possible dynamical complexity of these piecewise holomorphic models.

\subsection{Mixed anti‑holomorphic–holomorphic systems}
Systems of the form $\dot{z}=\overline{p(z)}$, where $p(z)$ is a polynomial of degree $n$, are 
especially interesting. Indeed, according to Theorem \ref{teoalg}, all their trajectories are algebraic 
since $\Omega=\int p(z)\,dz$ is a polynomial potential. Hence, we have that $\psi(x,y)=\Im(\Omega)$ is a polynomial first integral.

\begin{theorem}\label{teo_a}
Consider the piecewise planar dynamical system defined by
\[
z' =
\begin{cases}
\overline{p^+(z)}, & \text{if } \Im(z) > 0, \\
p^-(z), & \text{if } \Im(z) < 0,
\end{cases}
\]
where $z\in\mathbb{C}$, with
\[
p^+(z) = (a_1 + i a_2)z + (b_1 + i b_2), \quad
p^-(z) = (a + i b)(z-x_0),
\]
and $a_1,a_2,b_1,b_2,a,b,x_0\in\mathbb{R}$, $a_2\neq0$, $b\neq0$.
Then, the system has at most one limit cycle.

Moreover, this upper bound is sharp.
If such a limit cycle exists and $a\neq0$, then it is hyperbolic and its
stability is completely determined by the sign of $a$:
it is stable if $a<0$ and unstable if $a>0$.
\end{theorem}

\begin{proof}
Perform the translation $w=z-x_0$, which is a real horizontal shift since
$x_0\in\mathbb{R}$. Then $\Im(w)=\Im(z)$, so the switching line remains the real axis.

In the new variable, the system becomes
\[
w' =
\begin{cases}
\overline{p^+(w+x_0)}, & \text{if } \Im(w)>0,\\
(a+ib)w, & \text{if } \Im(w)<0.
\end{cases}
\]

The lower half-plane system is $\dot w=(a+ib)w$.
The upper half-plane system can be written as
\[
\dot w=\overline{(a_1+ia_2)w+(\tilde b_1+i\tilde b_2)},
\]
where $\tilde b_1=a_1x_0+b_1,\, \tilde b_2=a_2x_0+b_2.$ For the upper half-plane, define the complex potential
\[
\Omega^+(w)=\frac{A^+w^2}{2}+B^+w,
\quad A^+=a_1+ia_2,\quad B^+=\tilde b_1+i\tilde b_2.
\]
Its imaginary part
\[
\psi^+(x,y)=\Im\Omega^+(w)=\tilde{b}_2 x + \tilde{b}_1 y + \frac{a_2}{2} x^2 + a_1 x y - \frac{a_2}{2} y^2,\quad w=x+iy,
\]
is a first integral of the flow.
Therefore, any trajectory crossing the real axis at $(x_1,0)$ and $(x_2,0)$
satisfies
$
\psi^+(x_1,0)=\psi^+(x_2,0),
$
which yields
\[
(x_1-x_2)\left(\tilde b_2+\frac{a_2}{2}(x_1+x_2)\right)=0.
\]
Since $x_1\neq x_2$ for a nontrivial crossing, we obtain
\begin{equation}\label{rel1}
x_2=-x_1-\frac{2\tilde b_2}{a_2}.
\end{equation}

\medskip

In the lower half-plane, writing $w=re^{i\theta}$, the system becomes $\dot r=ar,\, \dot\theta=b,$
and admits the first integral
$
H^-(r,\theta)=b\ln r-a\theta.
$

Let $x_1$ and $x_2$ be the intersection points of a closed trajectory
with the real axis.
Since $r=|x|$ and $\theta=0$ for $x>0$, $\theta=\pi$ for $x<0$,
conservation of $H^-$ gives
\[
b\ln\left|\frac{x_2}{x_1}\right|=a(\theta_2-\theta_1).
\]

If $x_1$ and $x_2$ have the same sign, then $\theta_1=\theta_2$ and
$|x_1|=|x_2|$, which implies $x_1=x_2$, contradicting the existence of a
limit cycle.

If $x_1$ and $x_2$ have opposite signs, without loss of generality
$x_1>0$, $x_2<0$, so $\theta_1=0$, $\theta_2=\pi$, and
$\ln\left(-x_2/x_1\right)=a\pi/b.$

Combining this with \eqref{rel1}, we obtain
$
x_1=2\tilde b_2/a_2\left(e^{a\pi/b}-1\right),
$
which uniquely determines the crossing points.
Hence, the system has at most one limit cycle.
Since the translation $w=z-c$ is a diffeomorphism, the same holds for the original system. This upper bound is sharp, as illustrated in Example \ref{ex_cl_b}.

Now, assume that the limit cycle exists and let $\Pi:\mathbb{R}^+\to\mathbb{R}^+$
be the Poincar\'e return map defined on the positive real axis.
The map decomposes as $\Pi=F_L\circ F_U$, where $F_U$ corresponds to the
transition through the upper half-plane and $F_L$ to the transition through
the lower half-plane.

Let $x_1^*>0$ and $x_2^*<0$ denote the intersection points of the limit cycle
with the real axis, so that
$x_2^* = F_U(x_1^*), \, x_1^* = \Pi(x_1^*).$ From \eqref{rel1}, $F_U(x_1)=-x_1-\frac{2\tilde b_2}{a_2}$ and
$
F_U'(x_1^*)=-1.
$

In the lower half-plane, the explicit flow is
$r(t)=r_0e^{at}$, $\theta(t)=\theta_0+bt$.
Starting at $(x_2^*,0)$ with $x_2^*<0$, we have $r(0)=-x_2^*$ and $\theta(0)=\pi$.
The trajectory returns to the positive real axis at the first positive time
$T=\pi/|b|$.

At time $T$,
\[
F_L(x_2^*)=(-x_2^*)e^{a\pi/|b|},
\qquad
F_L'(x_2^*)=-e^{a\pi/|b|}.
\]

Therefore, the derivative of the Poincar\'e map at the fixed point $x_1^*$ is
\[
\Pi'(x_1^*) = F_L'(x_2^*)\,F_U'(x_1^*)
= e^{a\pi/|b|}.
\]

Since $|\Pi'(x_1^*)|\neq1$ for $a\neq0$, the limit cycle is hyperbolic.
Moreover, it is stable if $a<0$ and unstable if $a>0$.
\end{proof}

The previous result (Theorem~\ref{teo_a}) deals with the situation where the 
equilibrium point of the lower system lies on the switching line, i.e. \( z_0 = x_0 \in \mathbb{R} \). This restriction forces the lower flow to be rotationally symmetric with 
respect to a real center, leading to a particularly simple first integral and allowing a complete characterization of the limit cycle.

In contrast, the next theorem considers the general case where the equilibrium \( z_0 = x_0 + i y_0 \) of the lower vector field is an arbitrary point of the complex plane.  
Now \( z_0 \) is no longer constrained to lie on the switching line. It may be anywhere in the lower half‑plane (or even in the upper half‑plane, although the 
dynamics in \( \Im(z) < 0 \) is still governed by \( p^-(z) \)).  
This geometric freedom breaks the rotational symmetry of the lower flow, and the matching conditions between the two half‑planes become more intricate. As a consequence, the upper bound for the number of limit cycles increases, and the stability analysis, while still feasible, is considerably more involved.

The following theorem provides the optimal upper bound for this general configuration.

\begin{theorem}\label{teo_b}
Consider the piecewise planar dynamical system defined by:
\[
z' = 
\begin{cases} 
\overline{p^+(z)}, & \text{if } \Im(z) > 0, \\
p^-(z), & \text{if } \Im(z) < 0,
\end{cases}
\]
where \( z = x + iy \in \mathbb{C} \), with
\[
p^+(z) = (a_1 + i a_2)z + (b_1 + i b_2), \quad 
p^-(z) = (a + i b)(z - z_0),
\]
and \( a_1, a_2, b_1, b_2, a, b \in \mathbb{R} \), 
\( z_0 = x_0 + i y_0 \in \mathbb{C} \), 
\( a_2 \neq 0 \), \( b \neq 0 \).  
Then, the system has at most three limit cycles.
\end{theorem}

\begin{proof}
We treat separately the cases \(y_0 = 0\) and \(y_0 \neq 0\).

\medskip
\noindent\textbf{Case 1: \(y_0 = 0\) (real center).}  
Then \(z_0 = x_0\) is real, and the lower half-plane system is 
\(\dot{z} = (a+ib)(z-x_0)\).  
This situation was already analysed in Theorem~\ref{teo_a}, where it was shown
that there exists at most one limit cycle.

\medskip
\noindent\textbf{Case 2: \(y_0 \neq 0\) (complex center).}  
A limit cycle must intersect the real axis at two distinct points 
\(z_1 = x_1\) and \(z_2 = x_2\) with \(x_1 \neq x_2\).

\smallskip
\noindent\textit{Upper half-plane.}  
The upper system is Hamiltonian with Hamiltonian
\[
\psi^+(x,y) = \Im\!\Big(\frac{(a_1 + i a_2)z^2}{2} + (b_1 + i b_2)z\Big)
          = b_2 x + b_1 y + \frac{a_2}{2}x^2 + a_1 xy - \frac{a_2}{2}y^2 .
\]
Conservation of \(\psi^+\) along the upper arc gives
\[
\psi^+(x_1,0)=\psi^+(x_2,0) \Longrightarrow 
b_2 + \frac{a_2}{2}(x_1+x_2)=0,
\]
hence
\begin{equation}\label{eq:x2}
x_2 = -\,x_1 - \frac{2b_2}{a_2}.
\end{equation}
Define \(L(x) = -x - 2b_2/a_2\), then \(x_2 = L(x_1)\).

\smallskip
\noindent\textit{Lower half-plane.}  
Write \(w = z-z_0 = re^{i\theta}\).  The system becomes
$\dot r = a r,\, \dot\theta = b,$
which admits the first integral
\begin{equation}\label{eq:H}
H(r,\theta)=b\ln r - a\theta.
\end{equation}
On the real axis, \(w_j = x_j - x_0 - iy_0\) (\(j=1,2\)).  Thus
\[
r_j = \sqrt{(x_j-x_0)^2 + y_0^2},\qquad 
\theta_j = \arg(x_j-x_0 - iy_0).
\]

To handle the argument in a continuous way we employ the two-argument arctangent
\(\operatorname{atan2}(y,x)\), defined by
\[
\operatorname{atan2}(y,x)=
\begin{cases}
\arctan(y/x) & \text{if } x>0,\\
\arctan(y/x)+\pi & \text{if } x<0,\;y\ge 0,\\
\arctan(y/x)-\pi & \text{if } x<0,\;y<0,\\
\pi/2 & \text{if } x=0,\;y>0,\\
-\pi/2 & \text{if } x=0,\;y<0,
\end{cases}
\]
which returns the angle in \((-\pi,\pi]\).
Set
\[
\Theta(x)=\operatorname{atan2}(-y_0,\,x-x_0).
\]
For fixed \(y_0\neq0\), \(\Theta\) is a \(C^{\infty}\) function and satisfies
\begin{equation}\label{eq:ThetaPrime}
\Theta'(x)=\frac{y_0}{(x-x_0)^2+y_0^2}.
\end{equation}

Conservation of \(H\) along the lower arc yields
$b\ln(r_2/r_1)=a(\theta_2-\theta_1),$
i.e.
\begin{equation}\label{eq:lower}
b\ln\frac{\sqrt{(x_2-x_0)^2+y_0^2}}
{\sqrt{(x_1-x_0)^2+y_0^2}}
= a\big[\Theta(x_2)-\Theta(x_1)\big].
\end{equation}

Introduce \(R(x)=\sqrt{(x-x_0)^2+y_0^2}\).  
Using \eqref{eq:x2}, equation \eqref{eq:lower} becomes
\begin{equation}\label{eq:Fdef}
b\ln\frac{R(L(x))}{R(x)}
= a\big[\Theta(L(x))-\Theta(x)\big].
\end{equation}

Define
\[
F(x)=b\ln\frac{R(L(x))}{R(x)} 
- a\big[\Theta(L(x))-\Theta(x)\big].
\]

\smallskip
\noindent\textit{Derivative of \(F\).}  
From \eqref{eq:ThetaPrime}, \(L'(x)=-1\), 
\(R(x)R'(x)=x-x_0\) and 
\(R(L(x))R'(L(x))=L(x)-x_0\), we obtain
\begin{equation}\label{eq:Fprime}
\begin{aligned}
F'(x) &=
b\!\left[\frac{L(x)-x_0}{R(L(x))^2}L'(x)
-\frac{x-x_0}{R(x)^2}\right]
-a\!\left[\Theta'(L(x))L'(x)-\Theta'(x)\right] \\
&=
-\,b\!\left[\frac{L(x)-x_0}{R(L(x))^2}
+\frac{x-x_0}{R(x)^2}\right]
+a y_0\!\left[\frac{1}{R(L(x))^2}
+\frac{1}{R(x)^2}\right].
\end{aligned}
\end{equation}

Let \(u=x-x_0\), \(v=L(x)-x_0=-u-C\),
\(C=2x_0+\frac{2b_2}{a_2}\).
Multiplying \eqref{eq:Fprime} by 
\(D(x)=(u^2+y_0^2)(v^2+y_0^2)\) gives
\begin{equation}\label{eq:quadratic}
F'(x)D(x)=
(2a y_0-bC)u^{2}
+(2a y_0C-bC^{2})u
+(a y_0C^{2}+2a y_0^{3}+bC y_0^{2}).
\end{equation}

Equation \eqref{eq:quadratic} shows that the equation \(F'(x)=0\) is equivalent to a quadratic polynomial in \(u\), and therefore has at most two real solutions.  
Let \(m\) be the number of distinct real zeros of \(F\).  
If \(m \ge 4\), then by Rolle's theorem \(F'\) would have at least three distinct real zeros, contradicting the fact that \(F'\) has at most two real zeros.  
Thus \(m \le 3\).

Each real zero of \(F\) corresponds to a symmetric pair \((x_1, x_2 = L(x_1))\) that satisfies the matching conditions for a closed orbit, and hence to a limit cycle of the system. Consequently, the system possesses at most three limit cycles.

\medskip
\noindent
Combining both cases (\(y_0 = 0\) and \(y_0 \neq 0\)), we conclude that the piecewise system has at most three limit cycles. \qedhere
\end{proof}

\begin{example}\label{ex_cl_b}  \normalfont{Let us consider, for instance, a piecewise holomorphic system of the form
\[\dot{z}=\overline{p(z)},\quad\mbox{if}\quad \Im(z)>0;\quad\quad \dot{z}=-i(z-\beta)\quad\mbox{if}\quad \Im(z)<0\]
where
$p(z)=(a+bi)(z-ci),\quad a,b,c,\beta\in\R,c>0.$
According to the discussion above, 
$
\Omega= (a+bi)\,\frac{(z-ci)^2}{2},
$
is a complex potential of the system on $\{\Im z>0\}$. Hence a (polynomial) first integral is
$
\psi(x,y)=\Im\Omega
=  -ac\,x+bc\,y +   axy+\frac{b}{2}\bigl(x^{2}-y^{2}\bigr)-\frac{bc^{2}}{2}.
$
Since first integrals are defined up to an additive constant, we may equivalently take
$
 \psi(x,y)=-ac\,x+bc\,y +   axy+\frac{b}{2}\bigl(x^{2}-y^{2}\bigr).$
Note that $z_0=ci$ is an equilibrium point of $\dot{z}=\overline{p(z)}$. 
This equilibrium is a saddle point. Indeed, writing the system in real planar coordinates, we obtain
$$
x' = ax - by + bc, \quad y' = -bx - ay + ac,
$$
whose characteristic polynomial is
$
\lambda^2 - a^2 - b^2 = 0.
$
Thus, there exist two eigenvalues with opposite signs. The invariant directions are two straight lines $y=kx+c$, where
$
k=\frac{a \pm \sqrt{a^2+b^2}}{b}.
$
To illustrate, consider for instance $b=1$, $a=2$, $c=5$ and $\beta=10$. 
In this case we have
$
\psi(x,y)=-10x+5y+2xy+\frac{x^2}{2}-\frac{y^2}{2}.
$
The level $0$ intersects the $x$-axis at two points, $(0,0)$ and $(20,0)$. 
In the lower half-plane, the phase portrait is of center type. Since the center is exactly the midpoint of 
the segment between $x=0$ and $x=20$, it follows that there exists an orbit in $y<0$ connecting these two points. 
Thus, we obtain a cycle of the piecewise smooth system passing through the level set $\psi=0$, see Figure~\ref{figanti}.}\end{example}

\begin{figure}
	\centering 
	\includegraphics[width=7cm, height=5cm]{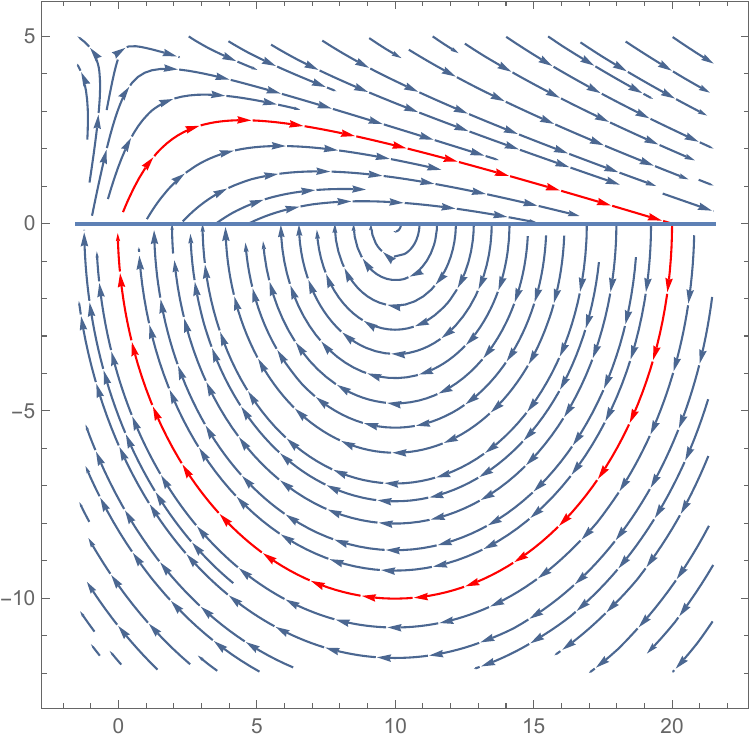} 
	\caption{Phase portrait of a piecewise holomorphic system, with the upper half-plane governed by an antiholomorphic polynomial.}\label{figanti}
\end{figure}

\subsection{Antiholomorphic polynomial systems}

We now turn our attention to a general class of systems where both the upper and lower half-planes are governed by arbitrary anti-holomorphic polynomials of the same degree. Our goal is to understand how the degree of these functions influences the number of limit cycles the system may exhibit.

We now consider the following piecewise anti-holomorphic polynomial system
\begin{equation}\label{main_eq}
\begin{aligned}
\left\{\begin{array}{l}
\dot{z}=\overline{p^+(z)}, \quad \Im(z)>0,\\[2pt]
\dot{z}=\overline{p^-(z)}, \quad \Im(z)<0,
\end{array} \right.
\end{aligned}
\end{equation}
where $z=x+iy$ and $p^\pm$ are complex polynomials.

Now, we establish a general result that bounds the number of limit cycles according to the degree of the polynomials.

\begin{theorem}\label{thm:limit_cycles}
For the piecewise anti-holomorphic polynomial system given by \eqref{main_eq}, the following upper bounds on the number of limit cycles hold:
\begin{enumerate}
    \item If \( p^+(z) \) and \( p^-(z) \) are linear polynomials, then the system has no limit cycles.
    \item If \( p^+(z) \) and \( p^-(z) \) are quadratic polynomials, then the system has at most one limit cycle.
    \item If \( p^+(z) \) and \( p^-(z) \) are cubic polynomials, then the system has at most three limit cycles.
\end{enumerate}
\end{theorem}
The proof of this result is developed through a sequence of propositions, each handling a specific degree case.
\begin{proposition}\label{prop12}
    Consider a piecewise anti-holomorphic polynomial system given by \eqref{main_eq}. If \( p^\pm \) are linear antiholomorphic functions, then the system described by \eqref{main_eq} has no limit cycles.
\end{proposition}
\begin{proof}
    Suppose that $p^\pm(z)=A^\pm z+B^\pm=u^\pm+iv^\pm,$ where $A^\pm=a_1^\pm+ia_2^\pm,B^\pm=b_1^\pm+ib_2^\pm\in\mathbb{C}.$  
     Now, we know that \[
\Omega^\pm(z)= \frac{A^\pm z^2}{2}+B^\pm z
\]
is a complex potential of the system on $\{\Im(z)>0\}$ and  $\{\Im (z)<0\}$, respectively. Thus, a first integral is 
$$\psi^\pm(x,y)=\Im\Omega^\pm(z)=b_2^\pm x+b_1^\pm y+\frac{a_2^\pm}{2} x^2+a_1^\pm xy-\frac{a_2^\pm y^2}{2},$$ and notice that $\partial\psi^\pm/\partial y = u^\pm$ and $\partial\psi^\pm/\partial x = v^\pm$, hence $\dot{x}=u^\pm$, $\dot{y}=-v^\pm$ is Hamiltonian. Now, define
        $$c^\pm(x_1,x_2)=\frac{\psi^\pm(x_1,0)-\psi^\pm(x_2,0)}{x_1-x_2}=\frac{2b_2^\pm+a_2^\pm(x_1+x_2)}{2},$$
        for $x_1\neq x_2.$ Recall that in the $(x_1, x_2)-$plane, $c^\pm(x_1,x_2)=0$ define implicitly the half return map in $\Sigma^\pm$. Given that \( c^\pm(x_1,x_2)=0 \) leads to the expression \( x_1=\frac{-2b_2^\pm-a_2^\pm x_2}{a_2^\pm} \), we observe that this represents a set of parallel straight lines. Consequently, the system defined by \( c^+(x_1,x_2)=c^-(x_1,x_2)=0 \) may either have no solutions for the variables \( x_1 \) and \( x_2 \) or possess infinitely many solutions. In both cases, the piecewise anti-holomorphic polynomial system cannot support the existence of limit cycles.

\end{proof}

\begin{proposition}\label{prop11}
    Consider a piecewise anti-holomorphic polynomial system given by \eqref{main_eq}. If \( p^\pm \) are quadratic antiholomorphic functions, then the system described by \eqref{main_eq} has at most one limit cycle.
\end{proposition}
\begin{proof}
    Suppose that $p^\pm(z)=A^\pm z^2+B^\pm z+C^\pm=u^\pm+iv^\pm,$ where $A^\pm=a_1^\pm+ia_2^\pm,B^\pm=b_1^\pm+ib_2^\pm,C^\pm=c_1^\pm+ic_2^\pm\in\mathbb{C}.$  
     Now, we know that \[
\Omega^\pm(z)= \frac{A^\pm z^3}{3}+\frac{B^\pm z^2}{2}+C^\pm z
\]
is a complex potential of the system on $\{\Im(z)>0\}$ and  $\{\Im (z)<0\}$, respectively. Thus, a first integral is 
$$\psi^\pm(x,y)=\Im\Omega^\pm(z)=
c_2^\pm x + \frac{b_2^\pm x^2}{2} + \frac{a_2^\pm x^3}{3} 
+ c_1^\pm y + b_1^\pm x y + a_1^\pm x^2 y 
- \frac{b_2^\pm y^2}{2} - a_2^\pm x y^2 - \frac{a_1^\pm y^3}{3},$$ and notice that $\partial\psi^\pm/\partial y = u^\pm$ and $\partial\psi^\pm/\partial x = v^\pm$, hence $\dot{x}=u^\pm$, $\dot{y}=-v^\pm$ is Hamiltonian. Now, define
        $$c^\pm(x_1,x_2)=\frac{\psi^\pm(x_1,0)-\psi^\pm(x_2,0)}{x_1-x_2}=c_2^\pm + \frac{b_2^\pm}{2}  (x_1 + x_2) + \frac{a_2^\pm}{3}  (x_1^2 + x_1x_2 + x_2^2),$$
        for $x_1\neq x_2.$ Recall that the polynomial \( c(x_1,x_2) \) is symmetric, so \( (x_1^0, x_2^0) \) is a solution of \( c^\pm(x_1, x_2) = 0 \) if, and only if, \( (x_2^0, x_1^0) \) is also a solution. Consequently, both solutions lead to the same potential limit cycle. Additionally, we know that if \( (x_1^0, x_2^0) \) satisfies \( c^+ (x_1^0, x_2^0) = c^- (x_1^0, x_2^0) = 0 \), then \( x_1^0 \) must be a root of the resultant polynomial \( R(x_1) \) of \( c^- \) and \( c^+ \) with respect to the variable \( x_2 \). Similarly, \( x_2^0 \) must be a root of the resultant polynomial \( S(x_2) \) of \( c^- \) and \( c^+ \) concerning \( x_1 \). Moreover, it is straightforward to see that \( R(x_1) = S(x_1) \). Thus, 
\[
\begin{aligned}
R(x_1) = \frac{1}{1296}\Big(&
-108 a_2^+ b_2^- b_2^+ c_2^- 
+ 108 a_2^- (b_2^+)^2 c_2^- 
+ 144 (a_2^+)^2 (c_2^-)^2 
+ 108 a_2^+ (b_2^-)^2 c_2^+\\
&- 108 a_2^- b_2^- b_2^+ c_2^+
- 288 a_2^- a_2^+ c_2^- c_2^+ 
+ 144 (a_2^-)^2 (c_2^+)^2
+ 72 (a_2^+)^2 b_2^- c_2^- x_1\\
&- 72 a_2^- a_2^+ b_2^+ c_2^- x_1
- 72 a_2^- a_2^+ b_2^- c_2^+ x_1
+ 72 (a_2^-)^2 b_2^+ c_2^+ x_1
+ 36 (a_2^+)^2 (b_2^-)^2 x_1^2 \\
&- 72 a_2^- a_2^+ b_2^- b_2^+ x_1^2
+ 36 (a_2^-)^2 (b_2^+)^2 x_1^2
\Big).
\end{aligned}
\]
 Thus, we conclude that \( R(x_1) \) can have at most 2 real roots. Due to the symmetry property, we can conclude that the piecewise anti-holomorphic polynomial system \eqref{main_eq} can have at most one limit cycle.

\end{proof}

\begin{proposition}\label{prop13}
    Consider a piecewise anti-holomorphic polynomial system given by \eqref{main_eq}. If \( p^\pm \) are cubic antiholomorphic functions, then the system described by \eqref{main_eq} has at most three limit cycles.
\end{proposition}
\begin{proof}
    Suppose that $p^\pm(z)=A^\pm z^3+B^\pm z^2+C^\pm z+D^\pm=u^\pm+iv^\pm,$ where $A^\pm=a_1^\pm+ia_2^\pm,B^\pm=b_1^\pm+ib_2^\pm,C^\pm=c_1^\pm+ic_2^\pm\in\mathbb{C}.$  
     Now, we know that \[
\Omega^\pm(z)= \frac{A^\pm z^4}{4}+\frac{B^\pm z^3}{3}+\frac{C^\pm z^2}{2}+D^\pm z
\]
is a complex potential of the system on $\{\Im(z)>0\}$ and  $\{\Im (z)<0\}$, respectively. Thus, a first integral is 
\begin{align*}
\psi^\pm(x,y) = \Im \Omega^\pm(z) 
&= d_2 x + \frac{c_2^{\pm} x^2}{2} + \frac{b_2^{\pm} x^3}{3} + \frac{a_2^{\pm} x^4}{4} + d_1^{\pm} y + c_1^{\pm} x y + b_1^{\pm} x^2 y + a_1^{\pm} x^3 y \\
&\quad - \frac{c_2^{\pm} y^2}{2} - b_2^{\pm} x y^2 - \frac{3}{2} a_2^{\pm} x^2 y^2 - \frac{b_1^{\pm} y^3}{3} - a_1^{\pm} x y^3 + \frac{a_2^{\pm} y^4}{4}
\end{align*}
and notice that $\partial\psi^\pm/\partial y = u^\pm$ and $\partial\psi^\pm/\partial x = v^\pm$, hence $\dot{x}=u^\pm$, $\dot{y}=-v^\pm$ is Hamiltonian. Now, define
        \begin{align*}c^\pm(x_1,x_2)&=d_2^{\pm} + \frac{c_2^{\pm}}{2}(x_1 + x_2) + \frac{b_2^{\pm}}{3}(x_1^2 + x_1 x_2 + x_2^2) + \frac{a_2^{\pm}}{4}(x_1^3 + x_1^2 x_2 + x_1 x_2^2 + x_2^3),\end{align*}
        for $x_1\neq x_2.$ Recall that the polynomial \( c(x_1,x_2) \) is symmetric, so \( (x_1^0, x_2^0) \) is a solution of \( c^\pm(x_1, x_2) = 0 \) if, and only if, \( (x_2^0, x_1^0) \) is also a solution. Consequently, both solutions lead to the same potential limit cycle. Additionally, we know that if \( (x_1^0, x_2^0) \) satisfies \( c^+ (x_1^0, x_2^0) = c^- (x_1^0, x_2^0) = 0 \), then \( x_1^0 \) must be a root of the resultant polynomial \( R(x_1) \) of \( c^- \) and \( c^+ \) with respect to the variable \( x_2 \). Similarly, \( x_2^0 \) must be a root of the resultant polynomial \( S(x_2) \) of \( c^- \) and \( c^+ \) concerning \( x_1 \). Moreover, it is straightforward to see that \( R(x_1) = S(x_1) \). Thus, 
\[R(x_1) = \frac{1}{2985984}
\sum_{i=0}^6\beta_i x^i,
\]
 where $\beta_i\in\mathbb{R},$ for all $i\in\{1,\dots,6\}$. We omit the explicit expressions for the coefficients $\beta_i$ as they are too lengthy. Thus, we conclude that \( R(x_1) \) can have at most 6 real roots. Due to the symmetry property, we can conclude that the piecewise anti-holomorphic polynomial system \eqref{main_eq} can have at most three limit cycles.
\end{proof}
 We now illustrate Proposition~\ref{prop11} with a concrete example of a piecewise anti-holomorphic polynomial system. 
 \begin{figure}[ht]
	\centering 
	\includegraphics[width=7cm, height=5cm]{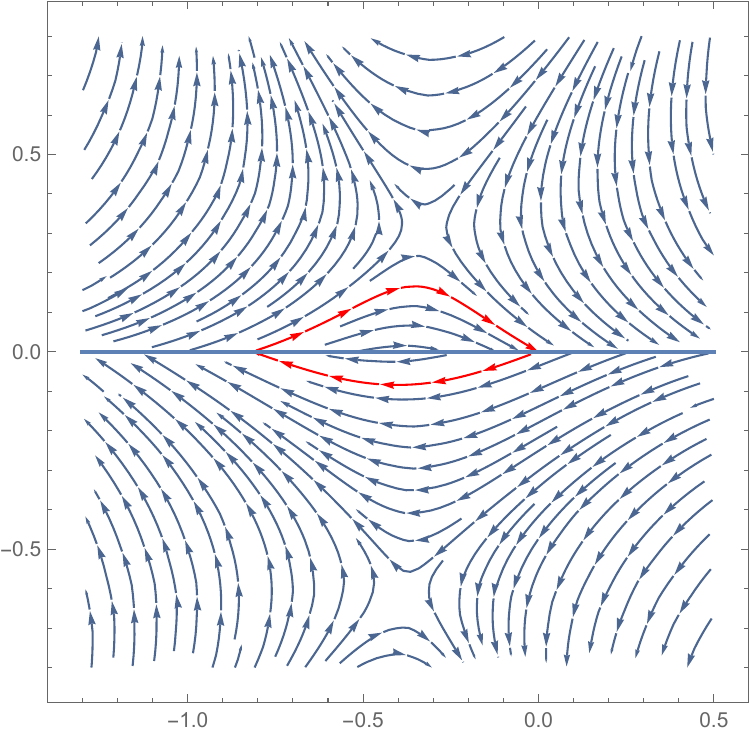} 
	\caption{Phase portrait of the piecewise anti-holomorphic polynomial system described in Example~\ref{example1}. The unique limit cycle obtained analytically is shown in the figure.}
	\label{figanti1}
\end{figure}
\begin{example}\label{example1}
Consider the piecewise anti-holomorphic polynomial system \eqref{thm:limit_cycles} defined by
\[
p^-(z) = (-3 + i) - (1 - i) z + \Bigg(-4 + \frac{i(-1 + \sqrt{33})}{-19 + 3\sqrt{33}}\Bigg) z^2, 
\qquad
p^+(z) = \frac{1}{2}\Big((2 + i) + (4 + 3 i) z + (6 + i) z^2\Big).
\]
The complex potentials of the system on $\{\Im(z)>0\}$ and  $\{\Im(z)<0\}$ are given, respectively, by 
\begin{align*}
\Omega^+(z)&=\frac{1}{2}(2 + i)z + \frac{1}{4}(4 + 3 i) z^2 + \frac{1}{6}(6 + i) z^3,\\
\Omega^-(z)&=(-3 + i)\, z - \left(\frac{1}{2} - \frac{i}{2}\right) z^2 + 
\frac{1}{3} \left(-4 + \frac{i(-1 + \sqrt{33})}{-19 + 3 \sqrt{33}}\right) z^3.
\end{align*}

Hence, we can explicitly compute the corresponding first integrals in each half-plane. These functions are
\[
\begin{aligned}
\varphi^+(x,y) &= \frac{x^3}{6} + y - \frac{3 y^2}{4} - y^3 + \frac{3}{4} x^2 (1 + 4 y)
- \frac{1}{2} x (-1 - 4 y + y^2),\\
\varphi^-(x,y)&= x + \frac{(-1 + \sqrt{33}) x^3}{-57 + 9 \sqrt{33}} + x^2 \left(\frac{1}{2} - 4 y\right)
- x y - \frac{(-1 + \sqrt{33}) x y^2}{-19 + 3 \sqrt{33}}
+ \frac{1}{6} y (-18 - 3 y + 8 y^2).
\end{aligned}
\]

These first integrals allow us to analyze the phase portrait and verify the existence of a limit cycle.  
In this particular example, we have
\[
\begin{aligned}
c^+(x_1,x_2) &= \frac{1}{12} \left(6 + 2 x_1^2 + 9 x_2 + 2 x_2^2 + x_1 (9 + 2 x_2)\right), \\[4pt]
c^-(x_1,x_2) &= 1 + \frac{x_1 + x_2}{2} + 
\frac{(-1 + \sqrt{33})(x_1^2 + x_1 x_2 + x_2^2)}{-57 + 9 \sqrt{33}}.
\end{aligned}
\]
Computing the resultant of \( c^+(x_1,x_2) \) and \( c^-(x_1,x_2) \) with respect to \( x_2 \), we obtain
\[
R(x_1) = \frac{20736 x_1 - 2304 \sqrt{33}\, x_1 + 9216 x_1^2}{5184 (-19 + 3 \sqrt{33})^2}.
\]
The real roots of this polynomial are
$
x_1 = 0, \, x_1 = \tfrac{1}{4}(-9 + \sqrt{33}).
$
Therefore, the system possesses a unique limit cycle, see Figure~\ref{figanti1}.
\end{example}

\section{Application: principal lines on immersed surfaces with constant mean curvature}\label{sec:curvature}
The study of foliations defined by level curves of harmonic functions finds a natural 
connection with the classical theory of surfaces in differential geometry. 
In particular, principal curvature lines of immersed surfaces with constant mean curvature 
(CMC surfaces) provide a rich geometric setting in which complex potentials naturally appear.\\

Let $M \subset \mathbb{R}^3$ be a smooth oriented surface immersed by a map 
$\alpha:U \subset \mathbb{R}^2 \to \mathbb{R}^3$. At each point $p \in M$, the operator 
$S:T_pM \to T_pM$ is defined by $S(v)=-D_vN$, where $N$ is the Gauss map. 
The eigenvalues of $S$ are the principal curvatures $k_1, k_2$, and the corresponding 
eigen-directions define two orthogonal line fields on the surface. 
The integral curves of these line fields are called \emph{principal curvature lines}. 

The mean curvature of the immersion is defined by
\[
H = \frac{k_1+k_2}{2}.
\]


A point $p\in M$ is called an \emph{umbilic point} if $k_1(p)=k_2(p)$. 
At an umbilic  all directions are principal directions. Outside umbilics, the principal line fields are smooth 
and form two orthogonal foliations of $M$.

A local parametrization $\alpha(u,v)$ of the surface is said to be in \emph{isothermal coordinates} 
if the first fundamental form has the form
\[I = E(u,v)(du^2+dv^2).\]
We introduce a complex coordinate $z=u+iv$. The second fundamental form as
\[ II = \Re\big( Q(z)dz^2 \big) + H \, I.\]

Let $p \in M$ be an umbilic point of the immersion $\alpha:U \subset \mathbb{R}^2 \to \mathbb{R}^3$. 
Outside umbilic points, the principal curvature directions define two line fields on $M$, 
which are orthogonal and smooth. If $p$ is an isolated umbilic, the \textit{index} of  
$p$ is defined as the winding number of the principal direction field around $p$. 
Equivalently, if $\theta(s)$ denotes a continuous determination of the angle of 
a principal direction along a small closed loop around $p$, then the index of $p$ 
is given by
\[\operatorname{Ind}(p) = \frac{1}{2\pi}\big(\theta(1)-\theta(0)\big).\]

It follows that the index of an isolated umbilic is always a negative half-integer, 
that is, $\operatorname{Ind}(p) = -\tfrac{n}{2}$ for some $n \in \mathbb{N}$.

A fundamental result says that  level sets of the real 
and imaginary parts of $Q(z)$ describe foliations that coincide 
with the principal curvature lines away from umbilic points.

\begin{theorem}
Let $p \in M$ be an isolated umbilic point of a smooth immersion 
$\alpha:U \subset \mathbb{R}^2 \to \mathbb{R}^3$ with constant mean curvature $H$. 
Then there exist local isothermal coordinates $(u,v)$ around $p$, with $z=u+iv$, 
such that the associated holomorphic function $Q$ takes the form
\[
Q(z) = z^n,
\]
for some integer $n \geq 1$. Moreover, the index of the umbilic is related to $n$ by
\[
\operatorname{Ind}(p) = -\tfrac{n}{2}.
\]
\end{theorem}
For a proof of this theorem we refer to \cite{GS}.\\


Thus, the local behavior of principal curvature lines around an isolated umbilic of a CMC immersion 
is completely determined by a monomial holomorphic function $Q$.

\section{Acknowledgements} Gabriel Rondón is partially supported by FAPESP grant 2024/15612-6. Paulo R. da Silva is partially supported by ANR-23-CE40-0028, 
FAPESP grant 2023/02959-5,  FAPESP grant 2024/15612-6 
 and CNPq grant 302154/2022-1.

\end{document}